\RequirePackage{ifpdf} 
\ifpdf
\documentclass[10pt,pdftex]{amsart}
\else
\documentclass[10pt,dvips]{amsart}
\fi

\ifpdf
\fi
\usepackage[english]{babel}

\usepackage{geometry}
\usepackage{amsmath, amssymb, cancel}

\usepackage{graphics}
\usepackage{subfigure}
\usepackage{graphicx}
\usepackage{epsfig}
\usepackage{amsmath,amssymb,amsthm}
\usepackage{epstopdf}

\usepackage{framed} 
\usepackage{booktabs}

\newtheorem{thm}{Theorem}[section]
\newtheorem{cor}[thm]{Corollary}
\newtheorem{lem}[thm]{Lemma}

\newtheorem{rem}[thm]{Remark}

\usepackage{amsaddr}

\numberwithin{equation}{section}

\usepackage[
  hidelinks,
  bookmarks,
  backref=false,
  plainpages=false
]{hyperref}

\usepackage{placeins}

\begin{document}
\newcommand{\BX}{{\bf X}}
\newcommand{\cv}{{\cal V}}
\newcommand{\cW}{{\cal W}}
\newcommand{\co}{{\cal O}}

\renewcommand{\theequation}{\thesection.\arabic{equation}}
\def\@eqnnum{{\reset@font\rm (\theequation)}}

\def\abstract{
\advance \rightskip by 10mm
\advance \leftskip by 10mm
\vspace{-0.8em}
\noindent
\small{\bf Abstract.}
}
\def\endabstract{\par\normalsize\rm}

\def\Xint#1{\mathchoice
{\XXint\displaystyle\textstyle{#1}}%
{\XXint\textstyle\scriptstyle{#1}}%
{\XXint\scriptstyle\scriptscriptstyle{#1}}%
{\XXint\scriptscriptstyle\scriptscriptstyle{#1}}%
\!\int}
\def\XXint#1#2#3{{\setbox0=\hbox{$#1{#2#3}{\int}$}
\vcenter{\hbox{$#2#3$}}\kern-.5\wd0}}
\def\ddashint{\Xint=}
\def\dashint{\Xint-}

\def\a{\alpha}
\def\b{\beta}
\def\d{\delta}\def\D{\Delta}
\def\e{\epsilon}
\def\g{\gamma}\def\G{\Gamma}
\def\k{\kappa}
\def\lam{\lambda}\def\Lam{\Lambda}
\renewcommand\o{\omega}\renewcommand\O{\Omega}
\def\s{\sigma}\def\S{\Sigma}
\renewcommand\t{\theta}\def\vt{\vartheta}
\newcommand{\vphi}{\varphi}
\def\z{\zeta}

\newcommand{\tsigma}{\tilde{\s}}
\newcommand{\tbsigma}{\tilde{\bsigma}}
\def\te{\tilde{\e}}
\def\tu{\tilde{u}}

\newcommand{\bchi}{\mbox{\boldmath$\chi$}}
\newcommand{\bdelta}{\mbox{\boldmath$\delta$}}
\newcommand{\bepsilon}{\mbox{\boldmath$\epsilon$}}
\newcommand{\bfeta}{\mbox{\boldmath$\eta$}}
\newcommand{\bgamma}{\mbox{\boldmath$\gamma$}}
\newcommand{\bomega}{\mbox{\boldmath$\omega$}}
\newcommand{\bvphi}{\mbox{\boldmath$\varphi$}}
\newcommand{\bphi}{\mbox{\boldmath$\phi$}}
\newcommand{\bPhi}{\mbox{\boldmath$\Phi$}}
\newcommand{\bpsi}{\mbox{\boldmath$\psi$}}
\newcommand{\bPsi}{\mbox{\boldmath$\Psi$}}
\newcommand{\bsigma}{\mbox{\boldmath$\sigma$}}
\newcommand{\btau}{\mbox{\boldmath$\tau$}}
\newcommand{\bxi}{\mbox{\boldmath$\xi$}}
\newcommand{\brho}{\mbox{\boldmath$\rho$}}
\newcommand{\bbeta}{\mbox{\boldmath$\beta$}}
\newcommand{\bzeta}{\mbox{\boldmath$\zeta$}}

\def\bk{\boldsymbol{\kappa}}
\def\bmu{\boldsymbol\mu}
\def\bxi{\boldsymbol{\xi}}
\def\bz{\boldsymbol{\zeta}}

\def\ba{{\bf a}}
\def\bb{{\bf b}}
\def\bc{{\bf c}}
\def\be{{\bf e}}
\def\bff{{\bf f}}
\def\bg{{\bf g}}
\def\bn{{\bf n}}
\def\bp{{\bf p}}
\def\bq{{\bf q}}
\def\bs{{\bf s}}
\def\bt{{\bf t}}
\def\bu{{\bf u}}
\def\bv{{\bf v}}
\def\bw{{\bf w}}
\def\bx{{\bf x}}
\def\by{{\bf y}}
\def\bzz{{\bf z}}

\def\bD{{\bf D}}
\def\bE{{\bf E}}
\def\bF{{\bf F}}
\def\bH{{\bf H}}
\def\bJ{{\bf J}}
\def\bV{{\bf V}}
\def\bU{{\bf U}}
\def\bW{{\bf W}}
\def\bX{{\bf X}}
\def\bY{{\bf Y}}

\def\cA{{\cal A}}
\def\cC{{\cal C}}
\def\cD{{\cal D}}
\def\cE{{\cal E}}
\def\cF{{\cal F}}
\def\cG{{\cal G}}
\def\cI{{\cal I}}
\def\cJ{{\cal J}}
\def\cK{{\cal K}}
\def\cL{{\cal L}}
\def\cO{{\cal O}}
\def\cP{{\cal P}}
\def\cQ{{\cal Q}}
\def\cR{{\cal R}}
\def\cS{{\cal \Sigma}}
\def\cT{{\cal T}}
\def\cU{{\cal U}}
\def\cV{{\cal V}}

\def\scT{{_\cT}}
\def\sD{{_D}}
\def\sE{{_E}}
\def\sF{{_F}}
\def\sFz{{_{F_z}}}
\def\sK{{_K}}
\def\sI{{_I}}
\def\sb{{_b}}
\def\sN{{_N}}

\def\curl{{{\bf curl} \ }}
\def\rot{{\mbox{rot}\ }}
\def\BPI{{\bf \Pi}}

\def\cth{\cT_h}
\def\ctH{\cT_H}

\def\tJ{\tilde{\J}}

\def\hK{\widehat{K}}
\def\hx{\widehat{x}}
\def\hy{\widehat{y}}
\def\bhv{\widehat{\bv}}

\def\l{\ell}
\def\bl{\boldsymbol{\ell}}
\def\col{\colon}
\def\f12{\frac12}
\def\dfrac{\displaystyle\frac}
\def\dint{\displaystyle\int}
\def\nab{\nabla}
\def\p{\partial}
\def\sm{\setminus}
\def\dsum{\displaystyle\sum}
\newcommand{\pp}[2]{\frac{\partial {#1}}{\partial {#2}}}
\def\bzero{{\bf 0}}

\def\divv{\nab\cdot}
\def\divx{\nab_x\cdot}
\def\divtx{\nab_{t,x}\cdot}
\def\nabx{\nab_x}

\newcommand{\grad}{\nabla}
\newcommand{\curlt}{{\nabla \times}}
\newcommand{\gperp}{\nabla^{\perp}}
\newcommand{\gradt}{\nabla\cdot}

\def\forallqq{\quad\forall\,}
\def\aph{A^{1/2}}
\def\amh{A^{-1/2}}

\def\osc{{\rm osc \, }}

\def\Im{{\rm Im}}
\newcommand{\tr}{{\rm tr}}
\def\divvr{{\rm div}}
\def\curllr{{\rm curl}}
\def\curll{{\rm curl}}
\def\curl{{\bf curl}}
\newcommand{\bgrad}{{\bf grad}}
\newcommand\diam{\mathrm{diam\,}}
\renewcommand\Im{\mathrm{Im\,}}
\def\Span{\mbox{Span}}
\def\supp{\mbox{supp\,}}
\newcommand{\trace}{{\rm trace}}

\newcommand{\tri}{|\!|\!|}
\newcommand{\ljump}{\lbrack\!\lbrack}
\newcommand{\rjump}{\rbrack\!\rbrack}
\newcommand{\bdm}{\begin{displaymath}}
\newcommand{\edm}{\end{displaymath}}
\newcommand{\beq}{\begin{equation}}
\newcommand{\eeq}{\end{equation}}
\newcommand{\beqa}{\begin{eqnarray}}
\newcommand{\eeqa}{\end{eqnarray}}
\newcommand{\beqas}{\begin{eqnarray*}}
\newcommand{\eeqas}{\end{eqnarray*}}
\newcommand{\ul}{\underline}
\newcommand{\wh}{\widehat}
\newcommand{\la}{\langle}
\newcommand{\ra}{\rangle}

\newcommand{\Lt}{L^2(\Omega)}
\newcommand{\Lts}{L^2(\Omega)^2}
\newcommand{\Ltc}{L^2(\Omega)^3}
\newcommand{\Ho}{H^1(\Omega)}
\newcommand{\Hoh}{H^1(\wh{\Omega})}
\newcommand{\Hoi}{H^1(\Omega_i)}
\newcommand{\Hos}{H^1(\Omega)^2}
\newcommand{\Hoc}{H^1(\Omega)^3}
\newcommand{\Hoch}{H^1(\wh{\Omega})^3}
\newcommand{\Hoci}{H^1(\Omega_i)^3}
\newcommand{\Hoz}{H^1_0(\Omega)}
\newcommand{\Ht}{H^2(\Omega)}
\newcommand{\Hti}{H^2(\Omega_i)}
\newcommand{\Hts}{H^2(\Omega)^2}
\newcommand{\Htc}{H^2(\Omega)^3}
\newcommand{\Htz}{H^0(\Omega)}
\newcommand{\Hh}{H^{1/2}(\Gamma)}
\newcommand{\Hhi}{H^{1/2}(\Gamma_i)}
\newcommand{\Hmh}{H^{-1/2}(\Gamma)}
\newcommand{\Hdiv}{H(\divvr;\,\Omega)}
\newcommand{\Hdivh}{H(\divv;\,\wh \Omega)}
\newcommand{\hcurl}{H(\curl\,A;\,\Omega)}
\newcommand{\Hcurl}{H(\curll\,A;\,\Omega)}
\newcommand{\Hcrl}{H(\curll\,;\,\Omega)}
\newcommand{\hcrl}{H(\curl\,;\,\Omega)}
\newcommand{\Hcrlh}{H(\curll\,;\,\wh\Omega)}
\newcommand{\hcrlh}{H(\curl\,;\,\wh\Omega)}
\newcommand{\Wdiv}{\BW_0(\mbox{\divv}\,;\,\Omega)}
\newcommand{\Wcurl}{\BW_0(\mbox{\curl}\,A;\,\Omega)}
\newcommand{\WcrossV}{\BW \times V}

\def\grad{{\nabla}}

\def\calS{{\cal S}}
\def\calT{{\cal T}}
\def\cA{{\mathcal A}}
\def\cB{{\cal B}}
\def\cD{{\mathcal{D}}}

\def\cH{{\cal H}}
\def\ba{{\mathbf{a}}}
\def\bbz{{\mathbf{z}}}

\def\beps{{\mathbf{\epsilon}}}

\def\bq{{\mathbf{q}}}

\def\br{{\mathbf{r}}}
\def\bg{\mathbf{g}}

\def\cM{{\mathcal{M}}}
\def\cN{{\mathcal{N}}}
\def\cT{{\mathcal{T}}}
\def\cE{{\mathcal{E}}}
\def\cP{{\mathcal{P}}}
\def\cF{{\mathcal{F}}}

\def\cB{{\mathcal{B}}}
\def\cG{{\mathcal{G}}}

\def\cL{{\mathcal{L}}}
\def\cJ{{\mathcal{J}}}
\def\cK{{\mathcal{K}}}
\def\cV{{\mathcal{V}}}
\def\cW{{\mathcal{W}}}
\def\bP{{\mathbf{P}}}
\def\bS{{\mathbf{S}}}
\def\bQ{{\mathbf{Q}}}

\def\bRT{{\mathbf{RT}}}
\def\bBDM{{\mathbf{BDM}}}

\def\bSigma{{\mathbf{\Sigma}}}

\def\R{\mathbb{R}}
\newcommand{\lJump}{[\![}
\newcommand{\rJump}{]\!]}
\newcommand{\jump}[1]{[\![ #1]\!]}

\newcommand{\sd}{\bsigma^{\Delta}}
\newcommand{\rd}{\brho^{\Delta}}

\newcommand{\eps}{\epsilon}
\newcommand{\dual}[2]{\left\langle #1,#2\right\rangle}

\def\ma{{\mathtt a}}
\def\mb{{\mathtt b}}
\def\mc{{\mathtt c}}
\def\md{{\mathtt d}}

\def\eg{{\mathtt {eg}}}
\def\full{{\mathtt {full}}}

\newcommand{\LS}{{\mathtt{LS}}}
\newcommand{\NLS}{{\mathtt{NLS}}}
\newcommand{\NJ}{{\mathtt{NJ}}}
\newcommand{\JJ}{{\mathtt{J}}}

\title [Goal-Oriented LSFEM]{Goal-Oriented Error Estimation for Least-Squares Finite Element Methods via Physically Meaningful Adjoint PDEs}
\author[Y. Wu and S. Zhang]{Yueyao Wu and Shun Zhang}
\address{Department of Mathematics, City University of Hong Kong, Kowloon Tong, Hong Kong, China}
\email{wyueyao2-c@my.cityu.edu.hk, shun.zhang@cityu.edu.hk}
\thanks{This work was supported in part by
Research Grants Council of the Hong Kong SAR, China, under the GRF Grant Project No. CityU 11316222, CityU 11305025}
\date{\today}


\maketitle

\begin{abstract}
We develop a goal-oriented error-estimation framework for first-order system least-squares (FOSLS) finite element methods based on the physical PDE adjoint rather than the adjoint induced by the least-squares formulation. For a prescribed output, the physical adjoint, identified here explicitly from the differential operator, its mixed boundary conditions, and the output, admits its own first-order flux system and hence a native, built-in least-squares estimator. The least-squares-induced adjoint, by contrast, restores Galerkin orthogonality in the least-squares inner product but corresponds to no simple boundary value problem, and does not inherit the native residual estimator of the physical first-order adjoint system. Our error identities rest on a primal--dual corrected functional and are \emph{formulation independent}: they follow only from the continuous primal and physical-adjoint equations, hold for arbitrary conforming approximations, and require no Galerkin orthogonality. The analysis covers linear second-order elliptic problems with mixed boundary conditions and quantities of interest involving volume, gradient, Neumann-boundary, and weighted Dirichlet-boundary-flux terms, the last determining the essential datum of the physical adjoint. We construct two corrected approximations of the output, one from the primal and adjoint potentials and one also using the fluxes, and show that both satisfy product-type error estimates in the primal and adjoint errors, hence converge at the combined primal--adjoint order. The native least-squares functionals then yield computable a~posteriori bounds, and their exact elementwise decomposition gives a balanced goal-oriented marking indicator whose global sum equals exactly the product of the two estimators. Numerical experiments confirm the predicted uniform-refinement rates and demonstrate robust product-rate behavior under adaptive refinement.
\end{abstract}

\section{Introduction}
\label{introduction}

In many applications the primary computational objective is not to
approximate the solution of a partial differential equation in a global
norm, but to evaluate a prescribed scalar output, a \emph{quantity of
interest} (QoI): a weighted average, a localized measurement, a boundary
trace, or a normal flux through part of the boundary. Goal-oriented error
estimation seeks both to assess the error in such an output and to steer
adaptive refinement toward the regions that most affect it. The adjoint,
or dual, problem is the classical instrument for this purpose: it encodes
the sensitivity of the output to local residuals and underlies the
duality-based and dual-weighted-residual a~posteriori theories
\cite{GS:02,PG:00,BR:01,BR:03}, as well as the convergence
analysis of goal-oriented adaptive finite element methods
\cite{MS:09,FPV:16}.

First-order system least-squares (FOSLS) finite element methods
\cite{CLMM:94,CMM:97,Jiang:98,BG:09,CS:04,
Ku:07,LSQoI:14,CFZ:15} recast
a high-order problem as a first-order system and
minimize the $L^2$ residual of that system. Under a norm equivalence
between the least-squares functional and the natural 
norm \cite{CLMM:94,CMM:97,CFZ:15,Zhang:23,Ku:07},
the value of the functional at the computed solution is a fully
computable, reliable, and efficient a~posteriori error estimator. This
\emph{built-in} estimator, which requires neither residual reconstruction
nor generic interpolation constants, is one of the principal attractions
of least-squares and, more generally, minimum-residual methods. For a series of problems, the built-in least-squares functional error estimators have been studied, for example, \cite{LZ:18,LZ:19,QZ:20}. For a
recent review and comparison of adaptive least-squares schemes we refer
to \cite{Bringmann:23}. The discontinuous Petrov--Galerkin (DPG) method
\cite{DPG:acta2025} is a closely related minimum-residual
method with an analogous built-in estimator \cite{CDG:14}.

Combining goal-oriented estimation with a least-squares method forces a
choice between two different notions of adjoint. The first is the
\emph{physical}, or PDE, adjoint: the boundary value problem determined
by the differential operator, its boundary conditions, and the output,
obtained directly from Green's formula. The second is a
\emph{formulation-induced} adjoint, obtained by transposing the bilinear
form of a particular discretization; for a least-squares formulation this
is the adjoint of the symmetric normal equations, an operator of the form
$L^{*}L$. The two are generally different, and the distinction is
decisive here. The physical adjoint carries the physical meaning of the
output, admits its own first-order flux system, and therefore possesses
its own built-in least-squares estimator. A formulation-induced adjoint,
by contrast, restores Galerkin orthogonality in the metric of the
formulation but corresponds to no simple boundary value problem. Related
constructions, differing among themselves, arise in the FOSLL$^{*}$
($LL^{*}$) methods \cite{CFZ:15} and in the DPG-star method
\cite{DGK:00}; none of them inherits the native residual estimator
associated with the physical first-order adjoint system.

This distinction is visible in existing approaches to QoI-oriented
least-squares and DPG methods. Chaudhry et~al.\ \cite{LSQoI:14}
enrich the least-squares functional with a QoI-tracking term and use an
auxiliary formulation-induced dual, solved in an enriched space, to
supply a computable proxy for the target value. Keith, Astaneh, and
Demkowicz \cite{KAD:19} develop goal-oriented DPG through a dual method
sharing the coefficient matrix of the primal DPG method but carrying a
different load, which restores orthogonality but requires a separately
constructed error estimator. In both cases the dual
actually used is the formulation-induced one rather than the physical
adjoint, and the native minimum-residual estimator is unavailable for it.
The question we address is whether one can retain the physical adjoint, and with it a genuine built-in estimator for \emph{both} the primal and the dual problems, while still obtaining a computable goal-oriented bound.

The obstacle is that a least-squares finite element solution is
orthogonal in the inner product induced by the least-squares functional
but is not the Galerkin projection with respect to the weak bilinear form
of the PDE; consequently the
standard goal-oriented identity, which relies on Galerkin orthogonality,
does not apply when the physical adjoint is used. We remove this obstacle
by means of a primal--dual corrected-functional value. The resulting
error identities are \emph{formulation independent}: they follow only
from the continuous primal and physical-adjoint equations, hold for
arbitrary conforming approximations, and require no Galerkin orthogonality
of the numerical method. The abstract primal--dual corrected functional
originates with Giles and S\"uli \cite{GS:02} (see also Pierce and
Giles \cite{PG:00}); our contribution is to realize it through
the \emph{physical} PDE adjoint of the mixed-boundary elliptic problem,
which we construct explicitly in Section~\ref{sec:primal-adjoint} and
combine with least-squares approximations of both the primal and the
adjoint problems. The least-squares method itself enters only in the
second part of the paper, as the discretization that renders the primal
and adjoint errors computable.

We carry out the analysis for linear second-order elliptic problems with
mixed Dirichlet--Neumann boundary conditions and quantities of interest
containing volume, gradient, Neumann-boundary, and weighted
Dirichlet-boundary-flux terms. The adjoint differential equation, its
mixed boundary conditions, and its first-order flux system are identified
directly from the primal problem and the physical output; in particular,
the weight appearing in the Dirichlet-boundary flux measurement becomes
the essential (Dirichlet) datum of the physical adjoint. Both the primal
and the adjoint problems are then written as first-order systems, which,
once discretized, yield two native least-squares functionals whose
residuals are equivalent to the corresponding flux--potential errors.

Within this framework we introduce two corrected approximations of the
output. The first, a potential-only corrected value, uses only the primal
and adjoint potentials. The second, a flux--potential corrected value,
additionally uses a flux approximation, which may be taken from the
least-squares method itself or from any other discretization, and which is
required only to be square integrable. For arbitrary conforming
approximations both satisfy exact error identities whose right-hand sides
are bilinear pairings of the primal and adjoint errors; the resulting
bounds are therefore products of the two, and the corrected outputs
converge at the sum of the primal and adjoint approximation orders. These
identities, like the framework itself, are independent of how the
underlying approximations are produced.

The least-squares method makes this framework computable. Because
the native primal and adjoint least-squares functionals control the
corresponding flux--potential errors, denoting their square roots by
$\eta_p^{\mathrm{LS}}$ and $\eta_d^{\mathrm{LS}}$, we obtain the
goal-oriented bounds
\[
|\gamma-\gamma_i|
\le C\,\eta_p^{\mathrm{LS}}\eta_d^{\mathrm{LS}},
\qquad i=0,1,
\]
computed entirely from the residuals of the physical primal and adjoint
first-order systems, and requiring neither Galerkin orthogonality for the
weak formulation nor a least-squares-induced adjoint. In contrast to the
existing approaches above, the physical adjoint here supplies a genuine
built-in estimator for the dual problem as well as the primal one.

The exact elementwise decomposition of each least-squares functional
supplies local information for adaptivity. A direct elementwise product of
the primal and adjoint indicators neither decomposes the global product
additively nor detects an element where only one of the two local
residuals is large. We therefore use a weighted goal-oriented marking of
the type introduced in \cite{BET:11}, here formed from the
built-in primal and adjoint least-squares functionals: the local primal
contributions are weighted by the global adjoint estimator and the local
adjoint contributions by the global primal estimator. With this balancing
the element indicators sum exactly to the product of the two global
estimators, so the resulting D\"orfler marking involves no unknown
estimator-equivalence constants in its relative primal--adjoint scaling.

The main contributions of this paper are as follows.
\begin{itemize}
\item We realize the Giles--S\"uli corrected-functional framework through
the physical PDE adjoint and separate this continuous construction from
its least-squares realization. This permits arbitrary conforming
approximations to be inserted into the corrected identities before the
least-squares estimators are invoked.
\item For mixed-boundary elliptic problems and quantities of interest
involving volume, gradient, Neumann-boundary, and Dirichlet-boundary flux
terms, we identify the physical adjoint problem and its first-order flux
formulation directly from the PDE and the output; the Dirichlet-boundary
flux weight becomes the essential datum of the physical adjoint.
\item We construct a potential-only and a flux--potential corrected
quantity and prove product-type error identities in the primal and
adjoint errors; both hold for arbitrary conforming approximations of the
potentials, the flux entering the second value being an arbitrary
square-integrable field, and both converge at the combined primal--adjoint
order.
\item Realizing the framework with FOSLS, we obtain computable
goal-oriented a~posteriori bounds in which the physical adjoint supplies a
built-in estimator for \emph{both} the primal and the dual problem, and we
derive the exact elementwise contributions and a balanced marking strategy
whose global sum equals the product of the two estimators.
\end{itemize}

The remainder of the paper is organized as follows.
Section~\ref{sec:model} introduces the primal boundary value problem, its
first-order formulation, and the class of quantities of interest.
Section~\ref{sec:primal-adjoint} derives the physical adjoint problem and
establishes the primal--adjoint representation of the output.
Section~\ref{sec:qoi-error-identities} introduces the two corrected
quantities and proves their product error identities; these three
sections are independent of the discretization.
Section~\ref{sec:LS-estimators} develops the primal and adjoint
least-squares estimators and the resulting goal-oriented bounds.
Section~\ref{sec:LSFEM} presents the least-squares finite element
discretizations and the discrete corrected outputs.
Section~\ref{sec:adaptive-algorithm} introduces the balanced adaptive
marking strategy. Numerical experiments are reported in
Section~\ref{sec:numerical}, and concluding remarks in
Section~\ref{sec:conclusions}.

\section{Model Elliptic Problem and Quantities of Interest}
\label{sec:model}

Let $\O\subset\mathbb R^d$, $d\ge2$, be a bounded Lipschitz domain with boundary decomposition $\partial\O=\overline{\G_D}\cup\overline{\G_N}$, where $\G_D$ and $\G_N$ are disjoint relatively open subsets of $\partial\O$ and $|\G_D|>0$.

We consider the second-order elliptic boundary value problem
\begin{equation}\label{eq_primal}
\begin{aligned}
-\gradt(A\nabla u)+\bb\cdot\nabla u+cu&=f_1+\gradt\bff_2 &&\text{in }\O,\\
u&=\phi_p &&\text{on }\G_D,\\
-(A\nabla u+\bff_2)\cdot\bn&=\psi_p &&\text{on }\G_N.
\end{aligned}
\end{equation}
Here $f_1\in L^2(\O)$, $\bff_2\in L^2(\O)^d$, $\phi_p\in H^{1/2}(\G_D)$, and $\psi_p\in L^2(\G_N)$. The space $H^{1/2}(\G_D)$ is understood as the restriction of $H^{1/2}(\partial\O)$ to $\G_D$, equipped with the quotient norm; in particular, every $\phi_p\in H^{1/2}(\G_D)$ admits an $H^1(\O)$ lifting.

The diffusion coefficient $A\in L^\infty(\O;\mathbb R^{d\times d})$ is symmetric and uniformly positive definite: there exist constants $0<a_0\le a_1<\infty$ such that
\begin{equation}\label{A_uniform_ellipticity}
a_0|\xi|^2\le \xi^TA(x)\xi\le a_1|\xi|^2
\qquad\forall\xi\in\mathbb R^d,\ \text{for a.e. }x\in\O.
\end{equation}
The lower-order coefficients satisfy $\bb\in L^\infty(\O)^d$ and $c\in L^\infty(\O)$. Define $H_D^1(\O):=\{v\in H^1(\O):v=0\text{ on }\G_D\}$. Since $|\G_D|>0$, the Poincar\'e inequality holds on $H_D^1(\O)$. Throughout, we write $\|v\|_1:=\|v\|_{H^1(\O)}$.

\subsection{The primal weak problem}
\label{subsec:primal-weak}

Choose a lifting $u_p\in H^1(\O)$ with $u_p=\phi_p$ on $\G_D$ and define
\begin{equation}\label{Xp}
X_p:=\{w\in H^1(\O):w=\phi_p\text{ on }\G_D\}=u_p+H_D^1(\O).
\end{equation}
For $t,s\in H^1(\O)$, define
\begin{equation}\label{def_a}
a(t,s):=(A\nabla t,\nabla s)+(\bb\cdot\nabla t+ct,s),
\end{equation}
and
\begin{equation}\label{def_ell}
\ell(s):=(f_1,s)-(\bff_2,\nabla s)-\langle\psi_p,s\rangle_{\G_N}.
\end{equation}
The boundary term is the $L^2(\G_N)$ pairing, and the trace theorem gives $\ell\in H^1(\O)'$. The coefficient assumptions imply
\begin{equation}\label{continuity_a}
|a(t,s)|\le C_{\rm con}\|t\|_1\|s\|_1
\qquad\forall t,s\in H^1(\O).
\end{equation}

Throughout, we assume the inf--sup condition
\begin{equation}\label{infsup}
0<\beta
=
\inf_{0\ne w_0\in H_D^1(\O)}
\sup_{0\ne v_0\in H_D^1(\O)}
\frac{a(w_0,v_0)}{\|w_0\|_1\|v_0\|_1}
=
\inf_{0\ne v_0\in H_D^1(\O)}
\sup_{0\ne w_0\in H_D^1(\O)}
\frac{a(w_0,v_0)}{\|w_0\|_1\|v_0\|_1}.
\end{equation}
Equivalently, the operator $\mathcal A:H_D^1(\O)\to H_D^1(\O)'$, defined by $\langle\mathcal Aw_0,v_0\rangle:=a(w_0,v_0)$, is an isomorphism.

The weak formulation of \eqref{eq_primal} is: find $u\in X_p$ such that
\begin{equation}\label{weak}
a(u,s_0)=\ell(s_0)
\qquad\forall s_0\in H_D^1(\O).
\end{equation}
Equivalently, writing $u=u_p+u_0$ with $u_0\in H_D^1(\O)$, we have
\begin{equation}\label{homo}
a(u_0,s_0)=\ell(s_0)-a(u_p,s_0)
\qquad\forall s_0\in H_D^1(\O).
\end{equation}

\begin{thm}\label{thm:primal_wellposed}
Under \eqref{infsup}, the weak problem \eqref{weak} has a unique solution $u\in X_p$.
\end{thm}

\begin{proof}
The result follows from the Babu\v{s}ka--Ne\v{c}as theorem applied to \eqref{homo}.
\end{proof}

\subsection{The primal first-order system}
\label{subsec:primal-fos}

For the weak solution $u\in X_p$, define the physical flux $\bsigma:=-A\nabla u-\bff_2$. Then $\bsigma\in L^2(\O)^d$, and the first equation of \eqref{eq_primal} yields $\gradt\bsigma=f_1-\bb\cdot\nabla u-cu\in L^2(\O)$; hence $\bsigma\in H(\divvr;\O)$. The associated first-order system is
\begin{equation}\label{fos_primal}
\begin{aligned}
\gradt\bsigma+\bb\cdot\nabla u+cu&=f_1 &&\text{in }\O,\\
\bsigma+A\nabla u&=-\bff_2 &&\text{in }\O,\\
u&=\phi_p &&\text{on }\G_D,\\
\bsigma\cdot\bn&=\psi_p &&\text{on }\G_N.
\end{aligned}
\end{equation}
The Neumann condition is understood in the normal-trace sense, namely,
\[
\langle\bsigma\cdot\bn,s_0\rangle_{\partial\O}
=
\langle\psi_p,s_0\rangle_{\G_N}
\qquad\forall s_0\in H_D^1(\O).
\]

\begin{lem}\label{lem:primal_weak_fos}
Let $u\in X_p$ and $\bsigma:=-A\nabla u-\bff_2$. Then $u$ satisfies \eqref{weak} if and only if $(\bsigma,u)$ satisfies \eqref{fos_primal}.
\end{lem}

\begin{proof}
Suppose that $u$ satisfies \eqref{weak}. As observed above, $\bsigma\in H(\divvr;\O)$ and $\gradt\bsigma+\bb\cdot\nabla u+cu=f_1$. The constitutive equation and the Dirichlet condition follow from the definition of $\bsigma$ and $u\in X_p$. For every $s_0\in H_D^1(\O)$, Green's formula gives
\[
a(u,s_0)
=
(\gradt\bsigma+\bb\cdot\nabla u+cu,s_0)
-(\bff_2,\nabla s_0)
-\langle\bsigma\cdot\bn,s_0\rangle_{\partial\O}.
\]
Comparison with \eqref{weak} yields $\bsigma\cdot\bn=\psi_p$ on $\G_N$ in the normal-trace sense.

Conversely, if $(\bsigma,u)$ satisfies \eqref{fos_primal}, then for every $s_0\in H_D^1(\O)$,
\[
a(u,s_0)
=
(f_1,s_0)-(\bff_2,\nabla s_0)-\langle\psi_p,s_0\rangle_{\G_N}
=
\ell(s_0),
\]
where Green's formula and the normal boundary condition have been used. Hence $u$ satisfies \eqref{weak}.
\end{proof}

Define
\begin{equation}\label{HNdiv}
H_N(\divvr;\O):=\{\btau\in H(\divvr;\O):\btau\cdot\bn=0\text{ on }\G_N\}.
\end{equation}
Assume that $\psi_p$ admits an $H(\divvr;\O)$ lifting $\bsigma_p$, i.e., $\bsigma_p\cdot\bn=\psi_p$ on $\G_N$, and define
\begin{equation}\label{Sigma_p}
\bSigma_p:=\bsigma_p+H_N(\divvr;\O)
=
\{\btau\in H(\divvr;\O):\btau\cdot\bn=\psi_p\text{ on }\G_N\}.
\end{equation}
Then $(\bsigma,u)\in\bSigma_p\times X_p$.

\subsection{General linear quantities of interest}
\label{subsec:primal-qoi}

Let $g_1\in L^2(\O)$, $\bg_2\in L^2(\O)^d$, $\psi_d\in L^2(\G_N)$, and $\phi_d\in H^{1/2}(\G_D)$. We consider quantities of interest involving volume and gradient observations, Neumann-boundary traces, and weighted Dirichlet-boundary fluxes, and define
\begin{equation}\label{cJp}
\cJ_p(w)
:=
(g_1,w)-(\bg_2,\nabla w)-\langle\psi_d,w\rangle_{\G_N}
-\langle(A\nabla w+\bff_2)\cdot\bn,\phi_d\rangle_{\G_D},
\end{equation}
whenever the final pairing is well defined; for the exact solution it is understood by duality through any 
 $H^1(\Omega)$ lifting of $\phi_d$, and is independent of the lifting by \eqref{weak}. In the present analysis, $\cJ_p$ is evaluated only at the exact primal solution; the corrected quantities introduced below are defined through 
$\ell$, $a(\cdot,\cdot)$, and the adjoint form  $m$ introduced in Section~\ref{sec:primal-adjoint},
and hence require no conforming flux for an arbitrary conforming potential approximation.

For the exact solution, $\bsigma=-A\nabla u-\bff_2$ gives
\[
-\langle(A\nabla u+\bff_2)\cdot\bn,\phi_d\rangle_{\G_D}
=
\langle\bsigma\cdot\bn,\phi_d\rangle_{\G_D},
\]
so \eqref{cJp} contains a weighted Dirichlet-boundary flux measurement. We define the scalar quantity of interest by 
\beq
\gamma:=\cJ_p(u).
\eeq 
As shown in the next section, the weight $\phi_d$ becomes the essential boundary datum of the physically meaningful adjoint problem.

\section{Physically Meaningful Primal--Adjoint PDEs}
\label{sec:primal-adjoint}

 The primal--dual framework of Giles and S\"uli~\cite{GS:02}
provides the algebraic structure underlying the corrected quantities
introduced below, but in an abstract form: the adjoint enters only
through a transposed bilinear form and an abstract functional. Our aim in
this section is to make this structure concrete and physical. For the
mixed-boundary elliptic problem \eqref{eq_primal} and the general
quantity of interest \eqref{cJp}, we determine explicitly the adjoint
differential operator, its Dirichlet and Neumann boundary conditions, and
its first-order flux system, directly from the primal PDE and the
physical output, so that the ``adjoint'' is itself a boundary value
problem one can write down, interpret, and discretize, rather than an
operator-level object tied to a particular formulation. In particular,
the weight $\phi_d$ of the Dirichlet-boundary flux measurement is
identified as the essential datum of the adjoint. This explicit,
PDE-level realization is what makes the adjoint physically meaningful and
what later allows a native least-squares estimator to be attached to it.

\subsection{The primal--dual framework}
\label{subsec:GS-framework}

Let $U$ and $V$ be Hilbert spaces, let $U_0\subset U$ and $V_0\subset V$ be closed subspaces, and define the affine spaces $U_p:=p+U_0$ and $V_d:=d+V_0$ for fixed $p\in U$ and $d\in V$. Let $B:U\times V\to\mathbb R$ be a bounded bilinear form, with $m\in U'$ and $\ell\in V'$. The primal measurement problem is to find $(\gamma_p,u)\in\mathbb R\times U_p$ such that
\begin{equation}\label{GS_primal}
\gamma_p=m(u)+\ell(v)-B(u,v)\qquad\forall v\in V_d,
\end{equation}
whereas the dual measurement problem is to find $(\gamma_d,z)\in\mathbb R\times V_d$ such that
\begin{equation}\label{GS_dual}
\gamma_d=m(w)+\ell(z)-B(w,z)\qquad\forall w\in U_p.
\end{equation}

For the elliptic problem of Section~\ref{sec:model}, we take $U=V=H^1(\O)$, $U_0=V_0=H_D^1(\O)$, $U_p=X_p$, and $B(\cdot,\cdot)=a(\cdot,\cdot)$, with $X_p$, $a$, and $\ell$ defined in \eqref{Xp}, \eqref{def_a}, and \eqref{def_ell}. It remains to identify the dual affine space $X_d$ and the functional $m$ from the physical quantity of interest $\cJ_p$.

\subsection{Identification of the adjoint data}
\label{subsec:adjoint-data}

We first identify the essential boundary data of the physical adjoint from
the boundary-flux contribution in the primal quantity of interest.
Let $z_d\in H^1(\O)$ be a function whose trace on $\G_D$ will be specified
below. Since
\[
A\nabla u+\bff_2=-\bsigma\in H(\divvr;\O),
\]
Green's formula gives
\begin{align}
a(u,z_d)
&=(A\nabla u,\nabla z_d)+(\bb\cdot\nabla u+cu,z_d)
=(A\nabla u+\bff_2,\nabla z_d)
  +(\bb\cdot\nabla u+cu,z_d)
  -(\bff_2,\nabla z_d)
\nonumber\\
&=\bigl(-\divvr(A\nabla u+\bff_2)
        +\bb\cdot\nabla u+cu,z_d\bigr)
  -(\bff_2,\nabla z_d)
  +\langle(A\nabla u+\bff_2)\cdot\bn,z_d\rangle_{\G_D}
  +\langle(A\nabla u+\bff_2)\cdot\bn,z_d\rangle_{\G_N}.
\label{aprimal}
\end{align}
Using the primal equation and the Neumann boundary condition
\[
-(A\nabla u+\bff_2)\cdot\bn=\psi_p
\qquad\text{on }\G_N,
\]
we obtain
\begin{align}
a(u,z_d)
&=(f_1,z_d)-(\bff_2,\nabla z_d)
  -\langle\psi_p,z_d\rangle_{\G_N}
  +\langle(A\nabla u+\bff_2)\cdot\bn,z_d\rangle_{\G_D}
=\ell(z_d)
  +\langle(A\nabla u+\bff_2)\cdot\bn,z_d\rangle_{\G_D}.
\label{aprimal_boundary}
\end{align}
Consequently,
\begin{equation}\label{ell_minus_a}
\ell(z_d)-a(u,z_d)
=
-\langle(A\nabla u+\bff_2)\cdot\bn,z_d\rangle_{\G_D}.
\end{equation}
The primal quantity of interest contains the Dirichlet-boundary flux term
\[
-\langle(A\nabla u+\bff_2)\cdot\bn,\phi_d\rangle_{\G_D}.
\]
Comparison with \eqref{ell_minus_a} therefore identifies the essential
boundary data of the physical adjoint:
\begin{equation}\label{dual_essential_data}
z_d=\phi_d
\qquad\text{on }\G_D.
\end{equation}
Thus, $z_d$ is chosen as an arbitrary $H^1(\O)$ lifting of $\phi_d$.

We next define
\begin{equation}\label{def_m}
m(w)
:=
(g_1,w)-(\bg_2,\nabla w)
-\langle\psi_d,w\rangle_{\G_N},
\qquad w\in H^1(\O).
\end{equation}
With the choice \eqref{dual_essential_data}, the primal quantity of
interest admits the representation
\begin{equation}\label{Jp_representation}
\cJ_p(u)
=
m(u)+\ell(z_d)-a(u,z_d).
\end{equation}
Define the adjoint affine space
\begin{equation}\label{Xd}
X_d:=\{v\in H^1(\O):v=\phi_d\text{ on }\G_D\}=z_d+H_D^1(\O).
\end{equation}
For any $v\in X_d$, the primal weak equation gives $\ell(v-z_d)=a(u,v-z_d)$; therefore,
\begin{equation}\label{primal_measurement_identity}
\cJ_p(u)=m(u)+\ell(v)-a(u,v)\qquad\forall v\in X_d.
\end{equation}
Thus the primal measurement problem is to find $(\gamma_p,u)\in\mathbb R\times X_p$ such that
\begin{equation}\label{problemP}
\gamma_p=m(u)+\ell(v)-a(u,v)\qquad\forall v\in X_d.
\end{equation}

\begin{thm}\label{thm:primal_QoI_equivalence}
The problem \eqref{problemP} has a unique solution $(\gamma_p,u)\in\mathbb R\times X_p$. Moreover, $u$ is the solution of \eqref{weak} and
\begin{equation}\label{gamma_p_identity}
\gamma_p=\cJ_p(u)=\gamma.
\end{equation}
\end{thm}

\begin{proof}
Subtracting \eqref{problemP} with $v=z_d$ from the same identity for arbitrary $v\in X_d$ gives $a(u,v-z_d)=\ell(v-z_d)$. Since $v-z_d$ is an arbitrary function in $H_D^1(\O)$, $u$ satisfies \eqref{weak} and is unique by Theorem~\ref{thm:primal_wellposed}. The scalar $\gamma_p$ is then uniquely determined, and \eqref{gamma_p_identity} follows from \eqref{Jp_representation}.
\end{proof}

\subsection{The adjoint problem and its first-order system}
\label{subsec:adjoint-problem}

The dual measurement problem associated with \eqref{problemP} is to find $(\gamma_d,z)\in\mathbb R\times X_d$ such that
\begin{equation}\label{problemD}
\gamma_d=m(w)+\ell(z)-a(w,z)\qquad\forall w\in X_p.
\end{equation}
Taking differences in $X_p$ shows that $z$ is characterized by the adjoint weak problem
\begin{equation}\label{weak_dual}
a(s_0,z)=m(s_0)\qquad\forall s_0\in H_D^1(\O).
\end{equation}
Since $z=\phi_d$ on $\G_D$, integration by parts in \eqref{weak_dual} gives the adjoint boundary value problem
\begin{equation}\label{eq_dual}
\begin{aligned}
-\gradt(A\nabla z+\bb z)+cz&=g_1+\gradt\bg_2 &&\text{in }\O,\\
z&=\phi_d &&\text{on }\G_D,\\
-(A\nabla z+\bb z+\bg_2)\cdot\bn&=\psi_d &&\text{on }\G_N.
\end{aligned}
\end{equation}

Define the adjoint flux $\br:=-(A\nabla z+\bb z+\bg_2)$. The associated first-order system is
\begin{equation}\label{fos_dual}
\begin{aligned}
\gradt\br+cz&=g_1 &&\text{in }\O,\\
\br+A\nabla z+\bb z&=-\bg_2 &&\text{in }\O,\\
z&=\phi_d &&\text{on }\G_D,\\
\br\cdot\bn&=\psi_d &&\text{on }\G_N.
\end{aligned}
\end{equation}
The Neumann condition is understood in the normal-trace sense, namely, $\langle\br\cdot\bn,s_0\rangle_{\partial\O}=\langle\psi_d,s_0\rangle_{\G_N}$ for all $s_0\in H_D^1(\O)$.

\begin{lem}\label{lem:dual_weak_fos}
Let $z\in X_d$ and $\br:=-(A\nabla z+\bb z+\bg_2)$. Then $z$ satisfies \eqref{weak_dual} if and only if $(\br,z)$ satisfies \eqref{fos_dual}.
\end{lem}

\begin{proof}
Suppose that $z$ satisfies \eqref{weak_dual}. Then $\br\in L^2(\O)^d$, and testing with compactly supported functions gives $\gradt\br=g_1-cz\in L^2(\O)$; hence $\br\in H(\divvr;\O)$ and the first equation of \eqref{fos_dual} holds. The constitutive and Dirichlet conditions follow from the definition of $\br$ and $z\in X_d$. For every $s_0\in H_D^1(\O)$, Green's formula gives
\[
a(s_0,z)=(\gradt\br+cz,s_0)-(\bg_2,\nabla s_0)-\langle\br\cdot\bn,s_0\rangle_{\partial\O},
\]
and comparison with \eqref{weak_dual} yields $\br\cdot\bn=\psi_d$ on $\G_N$ in the normal-trace sense.

Conversely, if $(\br,z)$ satisfies \eqref{fos_dual}, then for every $s_0\in H_D^1(\O)$,
\[
a(s_0,z)=(g_1,s_0)-(\bg_2,\nabla s_0)-\langle\psi_d,s_0\rangle_{\G_N}=m(s_0),
\]
so $z$ satisfies \eqref{weak_dual}.
\end{proof}

\begin{thm}\label{thm:dual_wellposed}
Under \eqref{infsup}, the adjoint weak problem \eqref{weak_dual} has a unique
solution $z\in X_d$.
\end{thm}

\begin{proof}
After writing $z=z_d+z_0$ with $z_0\in H_D^1(\O)$, the assertion follows
directly from the second inf--sup condition in \eqref{infsup}.
\end{proof}

Assume that $\psi_d$ admits an $H(\divvr;\O)$ lifting $\br_d$, i.e., $\br_d\cdot\bn=\psi_d$ on $\G_N$, and define
\begin{equation}\label{Sigma_d}
\bSigma_d:=\br_d+H_N(\divvr;\O)
=\{\brho\in H(\divvr;\O):\brho\cdot\bn=\psi_d\text{ on }\G_N\}.
\end{equation}
Then $(\br,z)\in\bSigma_d\times X_d$.

\subsection{The dual quantity of interest}
\label{subsec:dual-qoi}

Interchanging the primal and adjoint data suggests the dual physical functional
\begin{equation}\label{cJd}
\cJ_d(v):=(f_1,v)-(\bff_2,\nabla v)-\langle\psi_p,v\rangle_{\G_N}
-\langle(A\nabla v+\bb v+\bg_2)\cdot\bn,\phi_p\rangle_{\G_D},
\end{equation}
whenever the final pairing is well defined, understood as in \eqref{cJp}. In the present analysis, $\cJ_d$ is evaluated only at the exact adjoint solution; the corrected quantities introduced below require no conforming flux for an arbitrary conforming adjoint approximation. Since $\br=-(A\nabla z+\bb z+\bg_2)$,
\begin{equation}\label{Jd_exact}
\cJ_d(z)=\ell(z)+\langle\br\cdot\bn,\phi_p\rangle_{\G_D}.
\end{equation}

\begin{thm}\label{thm:dual_QoI_equivalence}
The dual measurement problem \eqref{problemD} has a unique solution $(\gamma_d,z)\in\mathbb R\times X_d$. Moreover, $z$ is the solution of \eqref{weak_dual} and
\begin{equation}\label{gamma_d_identity}
\gamma_d=\cJ_d(z).
\end{equation}
\end{thm}

\begin{proof}
Taking differences in \eqref{problemD} shows that $a(s_0,z)=m(s_0)$ for all $s_0\in H_D^1(\O)$; hence $z$ is the unique solution of \eqref{weak_dual} by Theorem~\ref{thm:dual_wellposed}, and $\gamma_d$ is uniquely determined. For any $w\in X_p$, Green's formula, $\gradt\br+cz=g_1$, $\br\cdot\bn=\psi_d$ on $\G_N$, and $w=\phi_p$ on $\G_D$ give
\[
a(w,z)=m(w)-\langle\br\cdot\bn,\phi_p\rangle_{\G_D}.
\]
Therefore,
\[
\gamma_d=m(w)+\ell(z)-a(w,z)
=\ell(z)+\langle\br\cdot\bn,\phi_p\rangle_{\G_D}
=\cJ_d(z),
\]
where \eqref{Jd_exact} was used.
\end{proof}

\subsection{Primal--adjoint equivalence}
\label{subsec:primal-adjoint-equivalence}

The following result is the analogue of the primal--dual equivalence theorem
of Giles and S\"uli \cite[Theorem~2.5]{GS:02}, written here for the
present mixed-boundary setting and the corresponding primal and adjoint
quantities of interest.

\begin{thm}[Primal--dual equivalence]\label{thm:primal_adjoint_equivalence}
Let $u\in X_p$ and $z\in X_d$ solve \eqref{weak} and
\eqref{weak_dual}, respectively. Then
\begin{equation}\label{primal_adjoint_equivalence}
\gamma=\gamma_p=\cJ_p(u)=\cJ_d(z)=\gamma_d,
\end{equation}
and their common value is
\begin{equation}\label{basic_gamma_representation}
\gamma=m(u)+\ell(z)-a(u,z).
\end{equation}
\end{thm}

\begin{proof}
Taking $v=z$ in \eqref{primal_measurement_identity} and $w=u$ in
\eqref{problemD}, the identities follow immediately.
\end{proof}

Thus, $\cJ_p(u)$ and $\cJ_d(z)$ are equivalent PDE representations of the same scalar output; the terms ``primal'' and ``adjoint'' only reflect the chosen direction of the construction.

\section{QoI Error Identities for General Conforming Approximations}
\label{sec:qoi-error-identities}

This section contains the core of the formulation-independent part of the
paper. Using the physically meaningful primal and adjoint problems of
Section~\ref{sec:primal-adjoint}, we construct two computable
\emph{corrected} approximations of the quantity of interest
\(\gamma=\cJ_p(u)=\cJ_d(z)\): a potential-only value $\gamma_0$
built from potential approximations, and a flux--potential value
$\gamma_1$ that additionally uses flux approximations. For each we prove an exact error identity,
Theorems~\ref{thm:gamma0_error} and \ref{thm:gamma1_error}, expressing
the output error as a bilinear pairing of the primal and adjoint
approximation errors, and hence a product-type bound. The decisive
feature is that these identities follow \emph{only} from the continuous
primal and adjoint equations: they hold for \emph{arbitrary} conforming
approximations of the potentials, require no Galerkin
orthogonality with respect to the physical bilinear form
$a(\cdot,\cdot)$, and make no reference to how the approximations are
produced. The least-squares method is not used anywhere in this section;
it enters only in Section~\ref{sec:LS-estimators}, where its built-in
functionals turn these identities into computable a~posteriori bounds. The
potential-only identity reformulates, in the present mixed-boundary
setting, the continuous error-correction framework of Giles and
S\"uli~\cite[Section~4]{GS:02}; the flux--potential corrected
quantity $\gamma_1$ and its product identity are, to our knowledge, new.

Let $u\in X_p$ and $z\in X_d$ denote the solutions of the primal and adjoint weak problems, respectively. Thus,
\begin{equation}\label{primal_weak_again}
a(u,s_0)=\ell(s_0)\qquad\forall s_0\in H_D^1(\O),
\end{equation}
and
\begin{equation}\label{dual_weak_again}
a(s_0,z)=m(s_0)\qquad\forall s_0\in H_D^1(\O).
\end{equation}
For arbitrary conforming approximations $(w,v)\in X_p\times X_d$, we have $u-w,z-v\in H_D^1(\O)$. Consequently, \eqref{primal_weak_again} and \eqref{dual_weak_again} imply
\begin{equation}\label{basic_error_identities}
a(u-w,z)=m(u-w),\qquad a(u,z-v)=\ell(z-v).
\end{equation}
These identities follow from the exact primal and adjoint problems, not from Galerkin orthogonality, and hold for every conforming pair $(w,v)\in X_p\times X_d$.

\subsection{A potential-only corrected quantity}
\label{subsec:gamma0}

The corrected-functional construction in this subsection follows the
general primal--dual error-correction framework of Giles and S\"uli
\cite[Section~4]{GS:02}.  We reformulate their continuous
error identity in the present mixed-boundary setting, using the
physical adjoint and the functionals introduced above.

For $(w,v)\in X_p\times X_d$, define the potential-only corrected quantity
\begin{equation}\label{def_gamma0}
\gamma_0(w,v):=m(w)+\ell(v)-a(w,v).
\end{equation}
By \eqref{basic_gamma_representation}, $\gamma_0(u,z)=\gamma$.
The following identity is the counterpart of
\cite[(4.1)]{GS:02} in the present notation.

\begin{thm}\label{thm:gamma0_error}
For every $(w,v)\in X_p\times X_d$,
\begin{equation}\label{first_error_identity}
\gamma-\gamma_0(w,v)=a(u-w,z-v).
\end{equation}
Consequently,
\begin{equation}\label{first_error_bound}
|\gamma-\gamma_0(w,v)|
\le C_{\rm con}\|u-w\|_1\|z-v\|_1.
\end{equation}
\end{thm}

\begin{proof}
Using \eqref{basic_gamma_representation}, \eqref{def_gamma0}, and the
primal and adjoint weak equations,
\[
m(u-w)=a(u-w,z),
\qquad
\ell(z-v)=a(u,z-v).
\]
Hence,
\[
\gamma-\gamma_0(w,v)
=
a(u-w,z)+a(u,z-v)-a(u,z)+a(w,v)
=
a(u-w,z-v),
\]
and \eqref{first_error_bound} follows from \eqref{continuity_a}.
\end{proof}

As in the general argument of Giles and S\"uli
\cite[Section~4]{GS:02}, the identity
\eqref{first_error_identity} is entirely continuous: it uses only the
exact primal and adjoint equations and does not require Galerkin
orthogonality of the approximations $w$ and $v$.

\subsection{A flux--potential corrected quantity}
\label{subsec:gamma1}
The potential-only quantity $\gamma_0$ uses only the potential
approximations $w,v$, and its error bound \eqref{first_error_bound}
involves only the $H^1$ potential errors. When flux approximations are
also at hand, from a first-order system least-squares method, a mixed
method, or a flux recovery applied to any other discretization, it is
natural to let the computed fluxes contribute to the output directly.
In this subsection we construct such a value, $\gamma_1$, and prove that
it too satisfies an exact product-type error identity
(Theorem~\ref{thm:gamma1_error}). The construction is purely continuous:
it uses no Galerkin orthogonality, and, as the proof shows, it places
no conformity requirement whatsoever on the flux approximations, which
enter only through $L^2$ pairings. Accordingly, throughout this
subsection
\[
w\in X_p,\qquad v\in X_d,\qquad \btau,\brho\in L^2(\O)^d,
\]
and the functionals \eqref{def_tilde_ell}--\eqref{def_tilde_a} below are
defined on $L^2(\O)^d\times H^1(\O)$. The least-squares method is invoked
only in Section~\ref{sec:LS-estimators}, to make the resulting bound
computable.
%

The exact constitutive equations imply
\[
\nabla u=-A^{-1}(\bsigma+\bff_2),
\qquad
\nabla z=-A^{-1}(\br+\bb z+\bg_2).
\]
Motivated by these identities, define
\begin{align}
\widetilde\ell(\brho,v)
&:=(f_1,v)
  +\bigl(\bff_2,A^{-1}(\brho+\bb v+\bg_2)\bigr)
  -\langle\psi_p,v\rangle_{\G_N},
\label{def_tilde_ell}\\
\widetilde m(\btau,w)
&:=(g_1,w)
  +\bigl(A^{-1}(\btau+\bff_2),\bg_2\bigr)
  -\langle\psi_d,w\rangle_{\G_N},
\label{def_tilde_m}\\
\widetilde a\bigl((\btau,w),(\brho,v)\bigr)
&:=\bigl(A^{-1}(\btau+\bff_2),\brho+\bg_2\bigr)
  +(cw,v).
\label{def_tilde_a}
\end{align}
Equivalently,
\begin{align}
\widetilde\ell(\brho,v)
&=\ell(v)
  +\bigl(\brho+\bg_2+A\nabla v+\bb v,A^{-1}\bff_2\bigr),
\label{tilde_ell_relation}\\
\widetilde m(\btau,w)
&=m(w)
  +\bigl(\btau+\bff_2+A\nabla w,A^{-1}\bg_2\bigr).
\label{tilde_m_relation}
\end{align}
For the exact pairs,
\[
\widetilde\ell(\br,z)=\ell(z),\qquad
\widetilde m(\bsigma,u)=m(u),\qquad
\widetilde a\bigl((\bsigma,u),(\br,z)\bigr)=a(u,z).
\]
The most direct flux--potential analogue of $\gamma_0$ is therefore
\begin{equation}\label{def_tilde_gamma1}
\widetilde\gamma_1(w,\btau,v,\brho)
:=
\widetilde m(\btau,w)+\widetilde\ell(\brho,v)
-\widetilde a\bigl((\btau,w),(\brho,v)\bigr).
\end{equation}

Although $\widetilde\gamma_1$ is consistent, it does not yet have a pure
product-error representation. Set
\[
E_p:=\bsigma-\btau,\qquad e_p:=u-w,\qquad
E_d:=\br-\brho,\qquad e_d:=z-v.
\]
A direct expansion, carried out in Appendix~\ref{app:expansion}, gives
\begin{align}
\gamma-\widetilde\gamma_1(w,\btau,v,\brho)
={}&
\bigl(A^{-1}E_p,E_d\bigr)+(ce_p,e_d)
+\bigl(\btau+A\nabla w+\bff_2,A^{-1}E_d\bigr)
+\bigl(A^{-1}E_p,\brho+\bg_2+A\nabla v+\bb v\bigr)
\nonumber\\
&
+\bigl(\btau+A\nabla w+\bff_2,A^{-1}\brho\bigr)
+\bigl(A^{-1}\btau,\brho+\bg_2+A\nabla v+\bb v\bigr).
\label{tilde_gamma1_error_expansion}
\end{align}
The first four terms are products of primal and adjoint approximation
errors. The last two are first-order defect terms: the first pairs the
primal constitutive defect with the approximate adjoint flux, while the
second pairs the adjoint constitutive defect with the approximate primal
flux. Since both terms are computable, we incorporate them into
the corrected quantity and define
\begin{align}
\gamma_1(w,\btau,v,\brho)
:={}&
\widetilde\gamma_1(w,\btau,v,\brho)
+\bigl(\btau+A\nabla w+\bff_2,A^{-1}\brho\bigr)
+\bigl(A^{-1}\btau,\brho+\bg_2+A\nabla v+\bb v\bigr).
\label{def_gamma1}
\end{align}
Thus, the two additional terms are precisely the computable defects
identified by the error expansion of $\widetilde\gamma_1$. Their
inclusion cancels the remaining first-order contributions and leaves
only products of primal and adjoint errors. In particular,
\[
\gamma_1(u,\bsigma,z,\br)=\gamma.
\]

\begin{thm}\label{thm:gamma1_error}
Let $w\in X_p$ and $v\in X_d$, and let $\btau,\brho\in L^2(\O)^d$ be
arbitrary. Then the corrected quantity \eqref{def_gamma1} satisfies
\begin{align}
\gamma-\gamma_1(w,\btau,v,\brho)
={}&
\bigl(A^{-1}(\bsigma-\btau),\br-\brho\bigr)
+\bigl(c(u-w),z-v\bigr)
\nonumber\\
&+\bigl(\btau+A\nabla w+\bff_2,
        A^{-1}(\br-\brho)\bigr)
+\bigl(A^{-1}(\bsigma-\btau),
        \brho+\bg_2+A\nabla v+\bb v\bigr).
\label{second_error_identity}
\end{align}
Equivalently,
\begin{align}
\gamma-\gamma_1(w,\btau,v,\brho)
={}&
-\bigl(A^{-1}(\bsigma-\btau),\br-\brho\bigr)
+\bigl(c(u-w),z-v\bigr)
\nonumber\\
&-\bigl(\nabla(u-w),\br-\brho\bigr)
-\bigl(\bsigma-\btau,
       \nabla(z-v)+A^{-1}\bb(z-v)\bigr).
\label{second_error_identity_2}
\end{align}
Consequently,
\begin{equation}\label{second_error_bound}
|\gamma-\gamma_1(w,\btau,v,\brho)|
\le C
\bigl(\|\bsigma-\btau\|_0+\|u-w\|_1\bigr)
\bigl(\|\br-\brho\|_0+\|z-v\|_1\bigr),
\end{equation}
with $C$ depending only on $a_0^{-1}$, $\|\bb\|_{0,\infty}$ and
$\|c\|_{0,\infty}$.
\end{thm}

\begin{proof}
The expansion \eqref{tilde_gamma1_error_expansion} of
$\gamma-\widetilde\gamma_1$, proved in Appendix~\ref{app:expansion},
uses only the exact constitutive relations
$\bsigma+A\nabla u+\bff_2=0$ and $\br+A\nabla z+\bb z+\bg_2=0$,
together with the primal and adjoint weak equations
\eqref{primal_weak_again}--\eqref{dual_weak_again} tested with
$u-w$ and $z-v$; these test functions lie in $H_D^1(\O)$ precisely
because $w\in X_p$ and $v\in X_d$. No further property of
$\btau,\brho$ is used. The definition \eqref{def_gamma1} of $\gamma_1$
cancels the last two terms of
\eqref{tilde_gamma1_error_expansion}, which gives
\eqref{second_error_identity}.

For \eqref{second_error_identity_2}, the same constitutive relations give
\[
\btau+A\nabla w+\bff_2
=-(\bsigma-\btau)-A\nabla(u-w),
\qquad
\brho+\bg_2+A\nabla v+\bb v
=-(\br-\brho)-A\nabla(z-v)-\bb(z-v).
\]
Substituting these into the third and fourth terms of
\eqref{second_error_identity} and using the symmetry of $A^{-1}$ yields
\eqref{second_error_identity_2}. Finally, \eqref{second_error_bound}
follows from the Cauchy--Schwarz inequality together with
\eqref{A_uniform_ellipticity} and the boundedness of $\bb$ and $c$.
\end{proof}

\begin{rem}\label{rem:gamma1_hypotheses}
The hypotheses of Theorem~\ref{thm:gamma1_error} are those actually used
in its proof. The fluxes occur in
\eqref{def_tilde_ell}--\eqref{def_gamma1} only through $L^2$ pairings:
neither $\gradt\btau$, $\gradt\brho$ nor their normal traces appear
there or anywhere in the argument, so square integrability is all that is
required of $\btau$ and $\brho$. The essential conditions $w=\phi_p$ and
$v=\phi_d$ on $\G_D$, by contrast, cannot be dispensed with: they place
$u-w$ and $z-v$ in $H_D^1(\O)$, which is what allows the primal and
adjoint weak equations to be tested with these functions in
Appendix~\ref{app:expansion}.

The stronger memberships $\btau\in\bSigma_p$ and $\brho\in\bSigma_d$
enter only in Section~\ref{sec:LS-estimators}, where the right-hand side
of \eqref{second_error_bound} is bounded by the least-squares
functionals: these involve $\gradt\btau$ and $\gradt\brho$, and the error
identities \eqref{primal_LS_error_identity},
\eqref{dual_LS_error_identity} use the Neumann conditions. The corrected
value $\gamma_1$ and its error identity are thus available under weaker
hypotheses than the built-in bound of
Theorem~\ref{thm:LS_QoI_estimates}.
\end{rem}

\subsection{Product-order convergence of corrected QoI values}
\label{subsec:product-order}

We now collect the approximation consequences of
Theorems~\ref{thm:gamma0_error} and \ref{thm:gamma1_error}. Let
$X_p^h\subset X_p$ and $X_d^H\subset X_d$ be conforming affine
approximation spaces and define
\[
\epsilon_{p,0}(h):=\inf_{w_h\in X_p^h}\|u-w_h\|_1,
\qquad
\epsilon_{d,0}(H):=\inf_{v_H\in X_d^H}\|z-v_H\|_1.
\]
For conforming flux--potential spaces
\[
\bSigma_p^h\times X_p^h\subset\bSigma_p\times X_p,
\qquad
\bSigma_d^H\times X_d^H\subset\bSigma_d\times X_d,
\]
set
\[
\epsilon_{p,1}(h):=
\inf_{(\btau_h,w_h)\in\bSigma_p^h\times X_p^h}
\bigl(\|\bsigma-\btau_h\|_0+\|u-w_h\|_1\bigr),
\]
and
\[
\epsilon_{d,1}(H):=
\inf_{(\brho_H,v_H)\in\bSigma_d^H\times X_d^H}
\bigl(\|\br-\brho_H\|_0+\|z-v_H\|_1\bigr).
\]

\begin{cor}\label{cor:qoi_rate}
Assume first that $u_h\in X_p^h$ and $z_H\in X_d^H$ satisfy
\[
\|u-u_h\|_1\le C\epsilon_{p,0}(h),
\qquad
\|z-z_H\|_1\le C\epsilon_{d,0}(H).
\]
Then
\begin{equation}\label{gamma0_rate_general}
|\gamma-\gamma_0(u_h,z_H)|
\le C\epsilon_{p,0}(h)\epsilon_{d,0}(H).
\end{equation}
If
\[
\epsilon_{p,0}(h)\le Ch^{s_{p,0}},
\qquad
\epsilon_{d,0}(H)\le CH^{s_{d,0}},
\]
then
\begin{equation}\label{gamma0_rate}
|\gamma-\gamma_0(u_h,z_H)|
\le Ch^{s_{p,0}}H^{s_{d,0}}.
\end{equation}

Assume next that
$(\bsigma_h,u_h)\in\bSigma_p^h\times X_p^h$ and
$(\br_H,z_H)\in\bSigma_d^H\times X_d^H$ satisfy
\[
\|\bsigma-\bsigma_h\|_0+\|u-u_h\|_1
\le C\epsilon_{p,1}(h),
\qquad
\|\br-\br_H\|_0+\|z-z_H\|_1
\le C\epsilon_{d,1}(H).
\]
Then
\begin{equation}\label{gamma1_rate_general}
|\gamma-\gamma_1(u_h,\bsigma_h,z_H,\br_H)|
\le C\epsilon_{p,1}(h)\epsilon_{d,1}(H).
\end{equation}
If
\[
\epsilon_{p,1}(h)\le Ch^{s_{p,1}},
\qquad
\epsilon_{d,1}(H)\le CH^{s_{d,1}},
\]
then
\begin{equation}\label{gamma1_rate}
|\gamma-\gamma_1(u_h,\bsigma_h,z_H,\br_H)|
\le Ch^{s_{p,1}}H^{s_{d,1}}.
\end{equation}
\end{cor}

\begin{proof}
The estimates follow directly from
Theorems~\ref{thm:gamma0_error} and \ref{thm:gamma1_error} and the
assumed quasi-optimality bounds.
\end{proof}

When the primal and adjoint problems are approximated on the same mesh,
$H=h$, the preceding estimates become
\[
|\gamma-\gamma_0(u_h,z_h)|
\le Ch^{s_{p,0}+s_{d,0}},
\qquad
|\gamma-\gamma_1(u_h,\bsigma_h,z_h,\br_h)|
\le Ch^{s_{p,1}+s_{d,1}}.
\]
Thus, the convergence order of each corrected quantity is the sum of
the corresponding primal and adjoint approximation orders. In
particular, equal primal and adjoint orders $s$ yield the rate
$O(h^{2s})$.

\begin{rem}\label{rem:qoi_rate_dof}
On a quasi-uniform mesh in $d$ dimensions, $N\simeq h^{-d}$, up to a
fixed factor accounting for the primal and adjoint unknowns. Hence an
$O(h^{s_p+s_d})$ corrected-output error corresponds to
\[
O\!\left(N^{-(s_p+s_d)/d}\right).
\]
In two dimensions, equal primal and adjoint orders $s$ therefore give
$O(N^{-s})$. Under adaptive refinement, the decay with respect to the
number of degrees of freedom depends on the simultaneous approximation
properties of the primal and adjoint solutions.
\end{rem}

\section{A Posteriori Error Estimation by Least-Squares Functionals}
\label{sec:LS-estimators}

In the preceding section, the errors in the corrected quantities were
bounded by products of the primal and adjoint approximation errors. In
this section, we show that the latter errors are controlled by the
least-squares functional estimators associated with the primal and
adjoint first-order systems. Combining the two estimates yields
computable a posteriori bounds for the corrected quantities of interest.

\subsection{Primal and adjoint least-squares functionals}
\label{subsec:LS-functionals}

For $(\btau,w)\in H(\divvr;\O)\times H^1(\O)$, define
\begin{equation}\label{product_norm}
\vert\!\vert\!\vert(\btau,w)\vert\!\vert\!\vert^2
:=\|\btau\|_{H(\divvr;\O)}^2+\|w\|_1^2.
\end{equation}
The least-squares functional associated with the primal first-order system \eqref{fos_primal} is
\begin{equation}\label{LS_functional_primal}
\begin{split}
\mathrm{LS}_p(\btau,w;f_1,\bff_2)
:={}&
\|A^{-1/2}\btau+A^{1/2}\nabla w+A^{-1/2}\bff_2\|_0^2
+\|\gradt\btau+\bb\cdot\nabla w+cw-f_1\|_0^2,
\end{split}
\end{equation}
while the adjoint functional associated with \eqref{fos_dual} is
\begin{equation}\label{LS_functional_dual}
\begin{split}
\mathrm{LS}_d(\brho,v;g_1,\bg_2)
:={}&
\|A^{-1/2}\brho+A^{1/2}\nabla v+A^{-1/2}\bb v+A^{-1/2}\bg_2\|_0^2
+\|\gradt\brho+cv-g_1\|_0^2.
\end{split}
\end{equation}
Both functionals are defined on $H(\divvr;\O)\times H^1(\O)$. Their homogeneous counterparts are $\mathrm{LS}_p(\btau_0,w_0;0,0)$ and $\mathrm{LS}_d(\brho_0,v_0;0,0)$.

\begin{thm}[LS norm equivalence]
\label{thm:LS_norm_equivalence}
Assume that \eqref{infsup} holds. Then there exist constants $c_{\rm LS},C_{\rm LS}>0$ such that
\begin{equation}\label{ls_equivalence_primal}
c_{\rm LS}\vert\!\vert\!\vert(\btau_0,w_0)\vert\!\vert\!\vert^2
\le \mathrm{LS}_p(\btau_0,w_0;0,0)
\le C_{\rm LS}\vert\!\vert\!\vert(\btau_0,w_0)\vert\!\vert\!\vert^2
\end{equation}
for all $(\btau_0,w_0)\in H_N(\divvr;\O)\times H_D^1(\O)$, and
\begin{equation}\label{ls_equivalence_dual}
c_{\rm LS}\vert\!\vert\!\vert(\brho_0,v_0)\vert\!\vert\!\vert^2
\le \mathrm{LS}_d(\brho_0,v_0;0,0)
\le C_{\rm LS}\vert\!\vert\!\vert(\brho_0,v_0)\vert\!\vert\!\vert^2
\end{equation}
for all $(\brho_0,v_0)\in H_N(\divvr;\O)\times H_D^1(\O)$.
\end{thm}

\begin{rem}\label{rem:LS_norm_equivalence}
The equivalences \eqref{ls_equivalence_primal} and \eqref{ls_equivalence_dual} are the standard stability estimates for the first-order least-squares formulations of the primal and adjoint problems, see \cite{CLMM:94,BLP:97,Cai:04,Ku:07,CZ:10b,CFZ:15}. A recent simple proof can be found in \cite{Zhang:23}. 
\end{rem}

\subsection{The primal least-squares functional estimator}
\label{subsec:primal-LS-estimator}

Let $(\btau,w)\in\bSigma_p\times X_p$. Since the exact pair $(\bsigma,u)$ satisfies \eqref{fos_primal}, subtraction gives
\begin{equation}\label{primal_LS_error_identity}
\mathrm{LS}_p(\btau,w;f_1,\bff_2)
=
\mathrm{LS}_p(\btau-\bsigma,w-u;0,0).
\end{equation}
Moreover, $\btau-\bsigma\in H_N(\divvr;\O)$ and $w-u\in H_D^1(\O)$. Define
\begin{equation}\label{eta_p}
\eta_p^{\mathrm{LS}}(\btau,w)
:=
\mathrm{LS}_p(\btau,w;f_1,\bff_2)^{1/2}.
\end{equation}

\begin{thm}\label{thm:primal_LS_estimator}
For every $(\btau,w)\in\bSigma_p\times X_p$,
\begin{equation}\label{primal_LS_estimator_equivalence}
c\vert\!\vert\!\vert(\bsigma-\btau,u-w)\vert\!\vert\!\vert
\le
\eta_p^{\mathrm{LS}}(\btau,w)
\le
C\vert\!\vert\!\vert(\bsigma-\btau,u-w)\vert\!\vert\!\vert.
\end{equation}
\end{thm}

\begin{proof}
The result follows directly from \eqref{primal_LS_error_identity} and the primal norm equivalence \eqref{ls_equivalence_primal}.
\end{proof}
\subsection{The adjoint least-squares functional estimator}
\label{subsec:dual-LS-estimator}

Let $(\brho,v)\in\bSigma_d\times X_d$. Since the exact pair $(\br,z)$ satisfies \eqref{fos_dual}, subtraction gives
\begin{equation}\label{dual_LS_error_identity}
\mathrm{LS}_d(\brho,v;g_1,\bg_2)
=
\mathrm{LS}_d(\brho-\br,v-z;0,0).
\end{equation}
Moreover, $\brho-\br\in H_N(\divvr;\O)$ and $v-z\in H_D^1(\O)$. Define
\begin{equation}\label{eta_d}
\eta_d^{\mathrm{LS}}(\brho,v)
:=
\mathrm{LS}_d(\brho,v;g_1,\bg_2)^{1/2}.
\end{equation}

\begin{thm}\label{thm:dual_LS_estimator}
For every $(\brho,v)\in\bSigma_d\times X_d$,
\begin{equation}\label{dual_LS_estimator_equivalence}
c\vert\!\vert\!\vert(\br-\brho,z-v)\vert\!\vert\!\vert
\le
\eta_d^{\mathrm{LS}}(\brho,v)
\le
C\vert\!\vert\!\vert(\br-\brho,z-v)\vert\!\vert\!\vert.
\end{equation}
\end{thm}

\begin{proof}
The result follows directly from \eqref{dual_LS_error_identity} and the adjoint norm equivalence \eqref{ls_equivalence_dual}.
\end{proof}

Theorems~\ref{thm:primal_LS_estimator} and \ref{thm:dual_LS_estimator} show that the primal and adjoint least-squares functionals are globally reliable and efficient estimators of the corresponding flux--potential errors.


\subsection{A posteriori estimates for the corrected QoI values}
\label{subsec:LS-QoI-estimates}

We now combine the product error estimates of Section~\ref{sec:qoi-error-identities} with the primal and adjoint least-squares functional estimators.

\begin{thm}\label{thm:LS_QoI_estimates}
Let $(\btau,w)\in\bSigma_p\times X_p$ and $(\brho,v)\in\bSigma_d\times X_d$. Then
\begin{equation}\label{LS_gamma0_estimate}
|\gamma-\gamma_0(w,v)|
\le C\,\eta_p^{\mathrm{LS}}(\btau,w)\eta_d^{\mathrm{LS}}(\brho,v),
\end{equation}
and
\begin{equation}\label{LS_gamma1_estimate}
|\gamma-\gamma_1(w,\btau,v,\brho)|
\le C\,\eta_p^{\mathrm{LS}}(\btau,w)\eta_d^{\mathrm{LS}}(\brho,v).
\end{equation}
\end{thm}

\begin{proof}
By Theorem~\ref{thm:gamma0_error},
\[
|\gamma-\gamma_0(w,v)|
\le C\|u-w\|_1\|z-v\|_1.
\]
Since
\[
\|u-w\|_1
\le \vert\!\vert\!\vert(\bsigma-\btau,u-w)\vert\!\vert\!\vert,
\qquad
\|z-v\|_1
\le \vert\!\vert\!\vert(\br-\brho,z-v)\vert\!\vert\!\vert,
\]
estimate \eqref{LS_gamma0_estimate} follows from \eqref{primal_LS_estimator_equivalence} and \eqref{dual_LS_estimator_equivalence}.

Similarly, Theorem~\ref{thm:gamma1_error} gives
\[
\begin{split}
|\gamma-\gamma_1(w,\btau,v,\brho)|
\le C&
\bigl(\|\bsigma-\btau\|_0+\|u-w\|_1\bigr)
\bigl(\|\br-\brho\|_0+\|z-v\|_1\bigr).
\end{split}
\]
Each factor is bounded by the corresponding product norm, so \eqref{primal_LS_estimator_equivalence} and \eqref{dual_LS_estimator_equivalence} yield \eqref{LS_gamma1_estimate}.
\end{proof}

\begin{rem}\label{rem:computable_QoI_estimators}
The estimates \eqref{LS_gamma0_estimate} and \eqref{LS_gamma1_estimate} are computable directly from the residuals of the physical primal and adjoint first-order systems. They require neither Galerkin orthogonality with respect to $a(\cdot,\cdot)$ nor replacement of the PDE adjoint by the adjoint induced by the least-squares bilinear form.
\end{rem}

\section{Least-Squares Finite Element Discretization and
Goal-Oriented Estimators}
\label{sec:LSFEM}

In this section, we specialize the least-squares functional estimators
of Section~\ref{sec:LS-estimators} to conforming least-squares finite
element approximations. Both the primal and adjoint problems are solved
on the same mesh. Since the corrected-output identities of Section~\ref{sec:qoi-error-identities}
hold for arbitrary conforming approximations, the least-squares finite
element solutions may be inserted directly into those identities. The
resulting QoI bounds are computable from the native residuals of the
physical primal and adjoint first-order systems.

\subsection{Finite element spaces}
\label{subsec:finite-element-spaces}

Let $\mathcal T_h$ be a shape-regular conforming simplicial triangulation of $\O$, possibly obtained by adaptive refinement; quasi-uniformity is assumed only when explicitly stated. For $K\in\mathcal T_h$, let $\mathbb P_k(K)$ denote the polynomials of total degree at most $k$ on $K$.

We use the lowest-order Raviart--Thomas space for the flux variables,
\begin{equation}\label{Sigma_h}
\bSigma_h:=\{\btau_h\in H(\divvr;\O):\btau_h|_K\in\mathbb P_0(K)^d+\bx\mathbb P_0(K)\ \forall K\in\mathcal T_h\},
\end{equation}
and the continuous piecewise linear space for the potential variables,
\begin{equation}\label{X_h}
X_h:=\{v_h\in C^0(\overline\O):v_h|_K\in\mathbb P_1(K)\ \forall K\in\mathcal T_h\}.
\end{equation}
The corresponding homogeneous boundary subspaces are
\begin{equation}\label{Sigma_h_N}
\bSigma_{h,N}:=\bSigma_h\cap H_N(\divvr;\O),
\qquad
X_{h,D}:=X_h\cap H_D^1(\O).
\end{equation}

This choice balances the principal approximation orders of the flux and potential components. On a quasi-uniform mesh, for sufficiently smooth functions,
\[
\inf_{\btau_h\in\bSigma_h}\|\bsigma-\btau_h\|_0\lesssim h,
\qquad
\inf_{w_h\in X_h}\|u-w_h\|_1\lesssim h.
\]
Thus, the $L^2$ approximation of the flux and the $H^1$ approximation of the potential have the same leading-order rate; the same applies to $(\br,z)$. Approximation in the full $H(\divvr;\O)$ norm additionally requires the corresponding regularity and approximation of the divergence.

Let $\mathcal T_0$ be the initial mesh, and let $X_0$ and $\bSigma_0$ denote the spaces \eqref{X_h} and \eqref{Sigma_h} associated with $\mathcal T_0$. For simplicity, assume that the boundary data are exactly representable on $\mathcal T_0$:
\begin{equation}\label{discrete_boundary_data_assumption}
\phi_p,\phi_d\in X_0|_{\G_D},
\qquad
\psi_p,\psi_d\in(\bSigma_0\cdot\bn)|_{\G_N}.
\end{equation}
Since all subsequent meshes are obtained by conforming refinement, the data remain representable on every refined mesh.

Define the discrete affine primal spaces
\begin{equation}\label{discrete_primal_affine_spaces}
X_p^h:=\{w_h\in X_h:w_h=\phi_p\text{ on }\G_D\},
\qquad
\bSigma_p^h:=\{\btau_h\in\bSigma_h:\btau_h\cdot\bn=\psi_p\text{ on }\G_N\},
\end{equation}
and the discrete affine adjoint spaces
\begin{equation}\label{discrete_dual_affine_spaces}
X_d^h:=\{v_h\in X_h:v_h=\phi_d\text{ on }\G_D\},
\qquad
\bSigma_d^h:=\{\brho_h\in\bSigma_h:\brho_h\cdot\bn=\psi_d\text{ on }\G_N\}.
\end{equation}
Then
\[
X_p^h\subset X_p,\qquad \bSigma_p^h\subset\bSigma_p,
\qquad
X_d^h\subset X_d,\qquad \bSigma_d^h\subset\bSigma_d.
\]

\subsection{The primal least-squares finite element method}
\label{subsec:primal-LSFEM}

The primal least-squares finite element approximation is the pair $(\bsigma_h,u_h)\in\bSigma_p^h\times X_p^h$ satisfying
\begin{equation}\label{discrete_primal_LS_min}
\mathrm{LS}_p(\bsigma_h,u_h;f_1,\bff_2)
=
\inf_{(\btau_h,w_h)\in\bSigma_p^h\times X_p^h}
\mathrm{LS}_p(\btau_h,w_h;f_1,\bff_2).
\end{equation}

For $(\brho,v),(\btau,w)\in H(\divvr;\O)\times H^1(\O)$, define
\begin{equation}\label{blsp}
\begin{split}
b_{{\rm ls},p}\bigl((\brho,v),(\btau,w)\bigr)
:={}&(\brho+A\nabla v,A^{-1}\btau+\nabla w)+(\gradt\brho+\bb\cdot\nabla v+cv,
\gradt\btau+\bb\cdot\nabla w+cw),
\end{split}
\end{equation}
and
\begin{equation}\label{Flsp}
F_{{\rm ls},p}(\btau,w)
:=-(\bff_2,A^{-1}\btau+\nabla w)
+(f_1,\gradt\btau+\bb\cdot\nabla w+cw).
\end{equation}
Then \eqref{discrete_primal_LS_min} is equivalent to find $(\bsigma_h,u_h)\in\bSigma_p^h\times X_p^h$, 
\begin{equation}\label{discrete_primal_LS_weak}
b_{{\rm ls},p}\bigl((\bsigma_h,u_h),(\btau_h,w_h)\bigr)
=
F_{{\rm ls},p}(\btau_h,w_h)
\qquad
\forall(\btau_h,w_h)\in\bSigma_{h,N}\times X_{h,D}.
\end{equation}
The bilinear form $b_{{\rm ls},p}$ is used only to compute the primal least-squares approximation; it is distinct from the physical weak form $a(\cdot,\cdot)$ and does not determine the physical adjoint problem.

By the primal least-squares norm equivalence,
\begin{equation}\label{coercivity_primal_LS}
b_{{\rm ls},p}\bigl((\btau_0,w_0),(\btau_0,w_0)\bigr)
=
\mathrm{LS}_p(\btau_0,w_0;0,0)
\ge
c_{\rm LS}\vert\!\vert\!\vert(\btau_0,w_0)\vert\!\vert\!\vert^2
\end{equation}
for all $(\btau_0,w_0)\in H_N(\divvr;\O)\times H_D^1(\O)$. Hence \eqref{discrete_primal_LS_weak} is uniquely solvable on every conforming mesh, without a mesh-size restriction.

\begin{thm}\label{thm:primal_LS_quasioptimal}
Let $(\bsigma,u)\in\bSigma_p\times X_p$ be the exact primal first-order solution, and let $(\bsigma_h,u_h)\in\bSigma_p^h\times X_p^h$ solve \eqref{discrete_primal_LS_min}. Then
\begin{equation}\label{primal_LS_quasioptimal}
\vert\!\vert\!\vert(\bsigma-\bsigma_h,u-u_h)\vert\!\vert\!\vert
\le
C\inf_{(\btau_h,w_h)\in\bSigma_p^h\times X_p^h}
\vert\!\vert\!\vert(\bsigma-\btau_h,u-w_h)\vert\!\vert\!\vert.
\end{equation}
Moreover,
\begin{equation}\label{eta_p_discrete}
c\vert\!\vert\!\vert(\bsigma-\bsigma_h,u-u_h)\vert\!\vert\!\vert
\le
\eta_p^{\mathrm{LS}}(\bsigma_h,u_h)
\le
C\vert\!\vert\!\vert(\bsigma-\bsigma_h,u-u_h)\vert\!\vert\!\vert.
\end{equation}
\end{thm}

\begin{proof}
The continuity and coercivity of $b_{{\rm ls},p}$, together with
$\bSigma_{h,N}\times X_{h,D}\subset H_N(\divvr;\O)\times H_D^1(\O)$,
give \eqref{primal_LS_quasioptimal}. Estimate \eqref{eta_p_discrete} follows from Theorem~\ref{thm:primal_LS_estimator} with $(\btau,w)=(\bsigma_h,u_h)$.
\end{proof}
\subsection{The adjoint least-squares finite element method}
\label{subsec:dual-LSFEM}

The adjoint least-squares finite element approximation is the pair $(\br_h,z_h)\in\bSigma_d^h\times X_d^h$ satisfying
\begin{equation}\label{discrete_dual_LS_min}
\mathrm{LS}_d(\br_h,z_h;g_1,\bg_2)
=
\inf_{(\brho_h,v_h)\in\bSigma_d^h\times X_d^h}
\mathrm{LS}_d(\brho_h,v_h;g_1,\bg_2).
\end{equation}

For $(\brho,v),(\btau,w)\in H(\divvr;\O)\times H^1(\O)$, define
\begin{equation}\label{blsd}
\begin{split}
b_{{\rm ls},d}\bigl((\brho,v),(\btau,w)\bigr)
:={}&
(\brho+A\nabla v+\bb v,
 A^{-1}\btau+\nabla w+A^{-1}\bb w)
 +
(\gradt\brho+cv,\gradt\btau+cw),
\end{split}
\end{equation}
and
\begin{equation}\label{Flsd}
F_{{\rm ls},d}(\btau,w)
:=
-(\bg_2,A^{-1}\btau+\nabla w+A^{-1}\bb w)
+(g_1,\gradt\btau+cw).
\end{equation}
Then \eqref{discrete_dual_LS_min} is equivalent to find $(\br_h,z_h)\in\bSigma_d^h\times X_d^h$, such that
\begin{equation}\label{discrete_dual_LS_weak}
b_{{\rm ls},d}\bigl((\br_h,z_h),(\btau_h,w_h)\bigr)
=
F_{{\rm ls},d}(\btau_h,w_h)
\qquad
\forall(\btau_h,w_h)\in\bSigma_{h,N}\times X_{h,D}.
\end{equation}

By the adjoint least-squares norm equivalence,
\begin{equation}\label{coercivity_dual_LS}
b_{{\rm ls},d}\bigl((\brho_0,v_0),(\brho_0,v_0)\bigr)
=
\mathrm{LS}_d(\brho_0,v_0;0,0)
\ge
c_{\rm LS}\vert\!\vert\!\vert(\brho_0,v_0)\vert\!\vert\!\vert^2
\end{equation}
for all $(\brho_0,v_0)\in H_N(\divvr;\O)\times H_D^1(\O)$. Hence the discrete adjoint problem is uniquely solvable on every conforming mesh, without a mesh-size restriction.

\begin{thm}\label{thm:dual_LS_quasioptimal}
Let $(\br,z)\in\bSigma_d\times X_d$ be the exact adjoint first-order solution, and let $(\br_h,z_h)\in\bSigma_d^h\times X_d^h$ solve \eqref{discrete_dual_LS_min}. Then
\begin{equation}\label{dual_LS_quasioptimal}
\vert\!\vert\!\vert(\br-\br_h,z-z_h)\vert\!\vert\!\vert
\le
C\inf_{(\brho_h,v_h)\in\bSigma_d^h\times X_d^h}
\vert\!\vert\!\vert(\br-\brho_h,z-v_h)\vert\!\vert\!\vert.
\end{equation}
Moreover,
\begin{equation}\label{eta_d_discrete}
c\vert\!\vert\!\vert(\br-\br_h,z-z_h)\vert\!\vert\!\vert
\le
\eta_d^{\mathrm{LS}}(\br_h,z_h)
\le
C\vert\!\vert\!\vert(\br-\br_h,z-z_h)\vert\!\vert\!\vert.
\end{equation}
\end{thm}

\begin{proof}
The proof is identical to that of Theorem~\ref{thm:primal_LS_quasioptimal}, using the continuity of $b_{{\rm ls},d}$, the coercivity estimate \eqref{coercivity_dual_LS}, and Theorem~\ref{thm:dual_LS_estimator}.
\end{proof}
\subsection{Discrete corrected QoI values}
\label{subsec:discrete-corrected-QoI}

Define
\begin{equation}\label{gamma0_h}
\gamma_{0,h}:=\gamma_0(u_h,z_h),
\qquad
\gamma_{1,h}:=\gamma_1(u_h,\bsigma_h,z_h,\br_h).
\end{equation}

\begin{cor}\label{cor:discrete_QoI_bounds}
The corrected finite element values satisfy
\begin{equation}\label{discrete_QoI_bounds}
|\gamma-\gamma_{0,h}|
\le
C\,\eta_p^{\mathrm{LS}}(\bsigma_h,u_h)
\eta_d^{\mathrm{LS}}(\br_h,z_h),
\end{equation}
and
\begin{equation}\label{discrete_QoI_bounds_gamma1}
|\gamma-\gamma_{1,h}|
\le
C\,\eta_p^{\mathrm{LS}}(\bsigma_h,u_h)
\eta_d^{\mathrm{LS}}(\br_h,z_h).
\end{equation}
\end{cor}
This follows immediately from Theorem~\ref{thm:LS_QoI_estimates} with
$(\btau,w)=(\bsigma_h,u_h)$ and $(\brho,v)=(\br_h,z_h)$.

Combining Corollary~\ref{cor:discrete_QoI_bounds} with Theorems~\ref{thm:primal_LS_quasioptimal} and \ref{thm:dual_LS_quasioptimal} gives, for $i=0,1$,
\begin{equation}\label{discrete_QoI_best_approximation}
\begin{split}
|\gamma-\gamma_{i,h}|
\le C&
\inf_{(\btau_h,w_h)\in\bSigma_p^h\times X_p^h}
\vert\!\vert\!\vert(\bsigma-\btau_h,u-w_h)\vert\!\vert\!\vert
\inf_{(\brho_h,v_h)\in\bSigma_d^h\times X_d^h}
\vert\!\vert\!\vert(\br-\brho_h,z-v_h)\vert\!\vert\!\vert.
\end{split}
\end{equation}

On a quasi-uniform family of meshes, suppose that
\begin{equation}\label{discrete_primal_adjoint_rates}
\inf_{(\btau_h,w_h)\in\bSigma_p^h\times X_p^h}
\vert\!\vert\!\vert(\bsigma-\btau_h,u-w_h)\vert\!\vert\!\vert
\lesssim h^{s_p},
\end{equation}
and
\begin{equation}\label{discrete_dual_rate}
\inf_{(\brho_h,v_h)\in\bSigma_d^h\times X_d^h}
\vert\!\vert\!\vert(\br-\brho_h,z-v_h)\vert\!\vert\!\vert
\lesssim h^{s_d}.
\end{equation}
Then
\begin{equation}\label{discrete_QoI_product_rate}
|\gamma-\gamma_{i,h}|
\lesssim h^{s_p+s_d},
\qquad i=0,1.
\end{equation}
In particular, under the regularity assumptions required for first-order approximation in the full flux--potential norm, the lowest-order Raviart--Thomas space and the continuous piecewise linear space give $s_p=s_d=1$, and hence
\[
|\gamma-\gamma_{i,h}|\lesssim h^2,
\qquad i=0,1.
\]

The estimates \eqref{discrete_primal_adjoint_rates}--\eqref{discrete_QoI_product_rate} apply to quasi-uniform meshes. On adaptively refined meshes, convergence is more naturally measured in terms of the number of degrees of freedom and the approximation classes of the primal and adjoint solutions.
\subsection{Elementwise least-squares contributions}
\label{subsec:elementwise-LS-contributions}

The primal and adjoint least-squares functionals admit exact elementwise decompositions. For each $K\in\mathcal T_h$, define
\begin{equation}\label{eta_p_K}
\begin{split}
(\eta_{p,K}^{\mathrm{LS}})^2
:={}&
\|A^{-1/2}\bsigma_h+A^{1/2}\nabla u_h+A^{-1/2}\bff_2\|_{0,K}^2+
\|\gradt\bsigma_h+\bb\cdot\nabla u_h+cu_h-f_1\|_{0,K}^2,
\end{split}
\end{equation}
and
\begin{equation}\label{eta_d_K}
\begin{split}
(\eta_{d,K}^{\mathrm{LS}})^2
:={}&
\|A^{-1/2}\br_h+A^{1/2}\nabla z_h
+A^{-1/2}\bb z_h+A^{-1/2}\bg_2\|_{0,K}^2+
\|\gradt\br_h+cz_h-g_1\|_{0,K}^2.
\end{split}
\end{equation}
Then
\begin{equation}\label{elementwise_LS_decomposition}
\bigl(\eta_p^{\mathrm{LS}}(\bsigma_h,u_h)\bigr)^2
=
\sum_{K\in\mathcal T_h}(\eta_{p,K}^{\mathrm{LS}})^2,
\qquad
\bigl(\eta_d^{\mathrm{LS}}(\br_h,z_h)\bigr)^2
=
\sum_{K\in\mathcal T_h}(\eta_{d,K}^{\mathrm{LS}})^2.
\end{equation}

Thus, \eqref{eta_p_K} and \eqref{eta_d_K} are the exact elementwise contributions to the primal and adjoint least-squares functionals, rather than local estimates involving unknown constants. They provide the local quantities used in the goal-oriented marking strategy of the next section.

\section{A Goal-Oriented Adaptive Algorithm}
\label{sec:adaptive-algorithm}

We turn the a~posteriori bounds of Section~\ref{sec:LSFEM} into a
goal-oriented marking strategy built directly from the physical primal and
adjoint first-order systems. The indicator below is a weighted marking of
the type introduced by Becker, Estecahandy, and
Trujillo~\cite{BET:11}, here formed from the
built-in primal and adjoint least-squares functionals rather than from
classical residual estimators. Write
$\eta_p^{\mathrm{LS}}:=\eta_p^{\mathrm{LS}}(\bsigma_h,u_h)$ and
$\eta_d^{\mathrm{LS}}:=\eta_d^{\mathrm{LS}}(\br_h,z_h)$, and recall that
the elementwise decompositions \eqref{elementwise_LS_decomposition} are
exact:
$(\eta_p^{\mathrm{LS}})^2=\sum_K(\eta_{p,K}^{\mathrm{LS}})^2$ and
$(\eta_d^{\mathrm{LS}})^2=\sum_K(\eta_{d,K}^{\mathrm{LS}})^2$.

The corrected-output estimate is of product form,
$|\gamma-\gamma_{i,h}|\le C\,\eta_p^{\mathrm{LS}}\eta_d^{\mathrm{LS}}$,
$i=0,1$. A na\"ive localization uses the elementwise product
$\eta_{p,K}^{\mathrm{LS}}\eta_{d,K}^{\mathrm{LS}}$, which is unsatisfactory:
it does not decompose the global product additively, since
$\sum_K\eta_{p,K}^{\mathrm{LS}}\eta_{d,K}^{\mathrm{LS}}\le
\eta_p^{\mathrm{LS}}\eta_d^{\mathrm{LS}}$ by Cauchy--Schwarz, and it
vanishes on any element where either factor vanishes, even if the other is
large, so an element important to only one problem may be overlooked.

To obtain an additive indicator we replace the product by a sum. For any
$\delta>0$, Young's inequality gives
\[
\eta_p^{\mathrm{LS}}\eta_d^{\mathrm{LS}}
\le\frac{\delta}{2}(\eta_p^{\mathrm{LS}})^2+\frac{1}{2\delta}(\eta_d^{\mathrm{LS}})^2,
\]
with equality precisely when the two terms coincide, i.e.\ when
$\delta=\eta_d^{\mathrm{LS}}/\eta_p^{\mathrm{LS}}$. Assuming
$\eta_p^{\mathrm{LS}},\eta_d^{\mathrm{LS}}>0$ and making this balancing
choice, the right-hand side becomes
$\tfrac12(\eta_d^{\mathrm{LS}}/\eta_p^{\mathrm{LS}})(\eta_p^{\mathrm{LS}})^2
+\tfrac12(\eta_p^{\mathrm{LS}}/\eta_d^{\mathrm{LS}})(\eta_d^{\mathrm{LS}})^2$;
substituting the exact decompositions and collecting elementwise motivates,
for each $K\in\mathcal T_h$, the \emph{balanced indicator}
\begin{equation}\label{adaptive_balanced_indicator}
\eta_K^2:=\frac12\left(\frac{\eta_d^{\mathrm{LS}}}{\eta_p^{\mathrm{LS}}}(\eta_{p,K}^{\mathrm{LS}})^2
+\frac{\eta_p^{\mathrm{LS}}}{\eta_d^{\mathrm{LS}}}(\eta_{d,K}^{\mathrm{LS}})^2\right) ,
\end{equation}
in which the local primal contribution is weighted by the global adjoint
estimator and vice versa. This is the least-squares analogue of the
weighted estimator of~\cite{BET:11}; in contrast to
the bounded weights $\zeta^2/(\eta^2+\zeta^2)$ used there, the balancing
weights $\tfrac12\,\eta_d^{\mathrm{LS}}/\eta_p^{\mathrm{LS}}$ and
$\tfrac12\,\eta_p^{\mathrm{LS}}/\eta_d^{\mathrm{LS}}$ make the element sum
equal the global product exactly, as we now record.

\begin{lem}\label{lem:adaptive_balanced_sum}
The indicators \eqref{adaptive_balanced_indicator} satisfy
$\sum_{K\in\mathcal T_h}\eta_K^2=\eta_p^{\mathrm{LS}}\eta_d^{\mathrm{LS}}$,
and hence, for $i=0,1$,
$|\gamma-\gamma_{i,h}|\le C\sum_{K\in\mathcal T_h}\eta_K^2$.
\end{lem}

\begin{proof}
Summing \eqref{adaptive_balanced_indicator} and using the exact
decompositions,
$\sum_K\eta_K^2
=\frac12\left(\frac{\eta_d^{\mathrm{LS}}}{\eta_p^{\mathrm{LS}}}(\eta_p^{\mathrm{LS}})^2
+\frac{\eta_p^{\mathrm{LS}}}{\eta_d^{\mathrm{LS}}}(\eta_d^{\mathrm{LS}})^2\right)
=\eta_p^{\mathrm{LS}}\eta_d^{\mathrm{LS}}$, the equality case of
Young's inequality at
$\delta=\eta_d^{\mathrm{LS}}/\eta_p^{\mathrm{LS}}$; the bound then follows
from the product estimate.
\end{proof}

The identity $\sum_K\eta_K^2=\eta_p^{\mathrm{LS}}\eta_d^{\mathrm{LS}}$ is
exact for the \emph{computed} least-squares functionals, so the relative
weighting in \eqref{adaptive_balanced_indicator} contains no unknown
estimator-equivalence constants; the stability constant $C$ enters only
the global reliability bound, not the marking. The balanced indicator also
removes the degeneracy of the pure product: if
$\eta_{d,K}^{\mathrm{LS}}=0$ but $\eta_{p,K}^{\mathrm{LS}}>0$ then
$\eta_K^2=\tfrac12(\eta_d^{\mathrm{LS}}/\eta_p^{\mathrm{LS}})(\eta_{p,K}^{\mathrm{LS}})^2>0$,
and symmetrically. We stress that $\eta_K^2$ is a marking weight tied to
the global product estimator, not an exact local contribution to the
signed error $\gamma-\gamma_{i,h}$.

Given $0<\theta<1$, we mark by the D\"orfler criterion, choosing
$\mathcal M_h\subset\mathcal T_h$ of (near-)minimal cardinality with
$\sum_{K\in\mathcal M_h}\eta_K^2\ge\theta\sum_{K\in\mathcal T_h}\eta_K^2$,
and refine the marked elements together with any needed for conformity;
the primal and adjoint problems are solved on the resulting common mesh,
which lets the local contributions be combined directly without
transferring indicators between triangulations. Because the weight of each
element combines the local primal and adjoint contributions, an element
that is important to either problem alone can still be selected, which the
elementwise product would not achieve. If either global estimator
vanishes, the corresponding flux--potential approximation is exact by the
norm equivalence
\eqref{ls_equivalence_primal}--\eqref{ls_equivalence_dual} and
$\gamma=\gamma_{i,h}$, so no further refinement is needed.

We do not claim a quasi-optimality result of the type
in~\cite{BET:11,MS:09}, whose
analyses proceed through Galerkin orthogonality in the energy norm and are
not available in the present least-squares setting. Here the balanced
indicator is used as a practical goal-oriented marking strategy, and the
experiments of Section~\ref{sec:numerical} illustrate the resulting
product-rate behavior of the corrected quantities of interest.

\section{Numerical Tests}
\label{sec:numerical}

We now illustrate the theory with four numerical experiments. In all
cases the primal and physical-adjoint first-order systems are discretized
by the least-squares finite element method of Section~\ref{sec:LSFEM},
using the lowest-order Raviart--Thomas space for the fluxes and continuous
piecewise linear elements for the potentials on the same mesh; the primal
and adjoint least-squares functionals serve as the built-in error
estimators $\eta_p^{\mathrm{LS}}$ and $\eta_d^{\mathrm{LS}}$, and the
corrected quantities $\gamma_{0,h}$ and $\gamma_{1,h}$ are evaluated from
the resulting discrete solutions. The experiments are designed to probe
the framework from complementary angles: Section~\ref{subsec:smooth-interface-example}
verifies, on a smooth problem with a discontinuous diffusion coefficient
across fitted interfaces, the primal--adjoint functional identity and the
predicted product-order convergence under uniform refinement;
Section~\ref{subsec:mommer-stevenson} uses a benchmark with spatially
separated primal and adjoint singularities, the configuration in which
purely multiplicative localization is most exposed, to test the balanced
marking strategy;
Section~\ref{subsec:convection-diffusion} treats a convection--diffusion
problem whose primal and adjoint boundary layers lie on opposite sides of
the domain; and Section~\ref{subsec:lshape} considers an L-shaped domain
with a reentrant-corner singularity probed by four different quantities of
interest. Throughout, $N$ denotes the number of mesh nodes, and all rates
are reported with respect to $N$. Non-polynomial Dirichlet data are imposed by
nodal interpolation in the potential spaces; the effect of boundary-data
approximation is not analyzed here. The reference values of the
quantities of interest used below are derived in
Appendix~\ref{app:reference-values}.

\subsection{A smooth interface problem with a convergence check}
\label{subsec:smooth-interface-example}
Let \(\Omega:=(-1,1)^2\), with subdomains
\(\Omega_1:=\{x+y>1\}\), \(\Omega_2:=\{x+y<-1\}\), and
\(\Omega_0:=\Omega\setminus\overline{\Omega_1\cup\Omega_2}\).
The diffusion tensor is \(A=\alpha I\) with \(\alpha=1\) in \(\Omega_0\) and \(\alpha=5\) in \(\Omega_1\cup\Omega_2\); we take \(\bb=(1,2)^T\), \(c=1\), \(\Gamma_N=\{y=1\}\cup\{x=-1\}\), and \(\Gamma_D=\{y=-1\}\cup\{x=1\}\). The initial triangulation resolves both interfaces \(x+y=\pm1\), and all subsequent meshes are obtained by conforming refinement.

\subsubsection{Manufactured primal and adjoint solutions}
For the primal problem we choose the smooth solution \(u=\cos(x+\tfrac12)\cos(y-\tfrac12)\) and set \(\bsigma:=-\alpha\nabla u-\bff_2\), with a piecewise divergence-free correction \(\bff_2\) (\(\bff_2=0\) in \(\Omega_0\)) making the normal flux continuous across the interfaces:
\[
\bff_2=
\begin{cases}
\bigl(4\sin(\tfrac32-y)\cos(y-\tfrac12),\ 4\cos(x+\tfrac12)\sin(\tfrac12-x)\bigr)^{T}, &\text{in }\Omega_1,\\[1mm]
\bigl(4\sin(-\tfrac12-y)\cos(y-\tfrac12),\ 4\cos(x+\tfrac12)\sin(-x-\tfrac32)\bigr)^{T}, &\text{in }\Omega_2.
\end{cases}
\]
By construction \(\divvr\bff_2=0\) in each \(\Omega_i\) and \([\bsigma\cdot n]=0\) across \(x+y=\pm1\), so \(\bsigma\in H(\divvr;\Omega)\); the primal data are \(f_1:=\divvr\bsigma+\bb\cdot\nabla u+cu\), \(\phi_p:=u|_{\Gamma_D}\), \(\psi_p:=\bsigma\cdot n|_{\Gamma_N}\).
For the adjoint problem we choose \(z=\cos(x-\tfrac12)\cos(y+\tfrac12)\) and \(\br:=-(\alpha\nabla z+\bb z+\bg_2)\), with the analogous correction (\(\bg_2=0\) in \(\Omega_0\))
\[
\bg_2=
\begin{cases}
\bigl(4\sin(\tfrac12-y)\cos(y+\tfrac12),\ 4\cos(x-\tfrac12)\sin(\tfrac32-x)\bigr)^{T}, &\text{in }\Omega_1,\\[1mm]
\bigl(4\sin(-y-\tfrac32)\cos(y+\tfrac12),\ 4\cos(x-\tfrac12)\sin(-x-\tfrac12)\bigr)^{T}, &\text{in }\Omega_2.
\end{cases}
\]
Since \(\bb\) is constant and \(z\) continuous, \(\bb z\) has no interface jump, and \(\bg_2\) cancels the normal jump of \(\alpha\nabla z\), so \([\br\cdot n]=0\) and \(\br\in H(\divvr;\Omega)\); the adjoint data are \(g_1:=\divvr\br+cz\), \(\phi_d:=z|_{\Gamma_D}\), \(\psi_d:=\br\cdot n|_{\Gamma_N}\).
The primal and dual representations of the output, evaluated at the exact
fields, agree by the continuous primal--adjoint identity and give the
reference value $\gamma_{\rm ref}=5.058351722069$; see
Appendix~\ref{app:reference-values}.

\subsubsection{Uniform-refinement convergence}
Under uniform refinement, with every mesh resolving the interfaces
\(x+y=\pm1\), hence the coefficient jumps, exactly,
Figure~\ref{fig:convergence-history} shows that both flux--potential
errors decay as \(O(N^{-1/2})\) in the
\(H(\divvr;\Omega)\times H^1(\Omega)\) norm. The built-in estimators
\(\eta_p^{\rm LS}\) and \(\eta_d^{\rm LS}\) track the primal and dual
errors almost exactly, so that their
product satisfies \(\eta_p^{\rm LS}\eta_d^{\rm LS}=O(N^{-1})\). Both
corrected outputs \(|\gamma_{\rm ref}-\gamma_{i,h}|\) then decay at this
product rate \(O(N^{-1})\). On these quasi-uniform meshes
\(N\simeq h^{-2}\), so that the observed rates \(O(N^{-1/2})\) and
\(O(N^{-1})\) correspond to \(O(h)\) and \(O(h^2)\), as predicted by
Corollary~\ref{cor:qoi_rate}. The example thus verifies the
discontinuous-coefficient treatment on fitted interfaces, the mixed
boundary conditions, the primal--adjoint identity, and the predicted
first- and second-order rates.

\begin{figure}[htbp]
\centering
\includegraphics[width=0.6\textwidth]{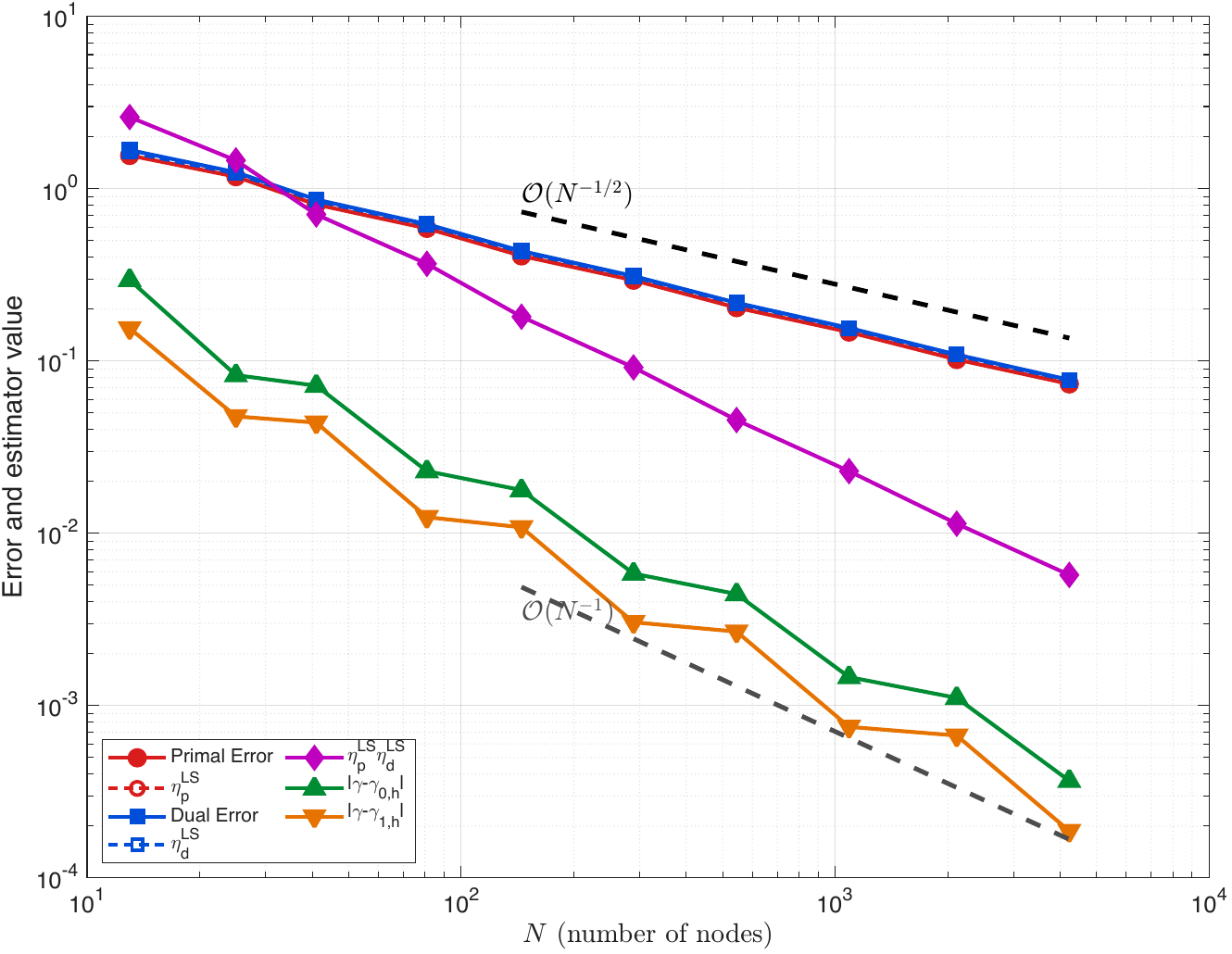}
\caption{Uniform-refinement convergence history for the smooth interface problem. The primal and adjoint least-squares errors and the built-in estimators \(\eta_p^{\rm LS}\), \(\eta_d^{\rm LS}\) decay as
\(O(N^{-1/2})\); the product estimator \(\eta_p^{\rm LS}\eta_d^{\rm LS}\)
and the corrected output errors \(|\gamma_{\rm ref}-\gamma_{0,h}|\),
\(|\gamma_{\rm ref}-\gamma_{1,h}|\) decay as \(O(N^{-1})\).}
\label{fig:convergence-history}
\end{figure}

\subsection{The Mommer--Stevenson benchmark}
\label{subsec:mommer-stevenson}
We next consider the goal-oriented benchmark of Mommer and Stevenson~\cite{MS:09}, in which the dominant local primal and adjoint residual contributions are concentrated in spatially separated regions, a configuration that triggers the ``blind spot'' of purely multiplicative (local product) marking. Let $\Omega:=(-1,1)^2$ and
\[
-\Delta u=\nabla\cdot\bff_2,\qquad -\Delta z=\nabla\cdot\bg_2\quad\text{in }\Omega,\qquad u=z=0\ \text{on }\partial\Omega,
\]
so $f_1=g_1=0$ and $\phi_p=\phi_d=0$. The sources are the piecewise-constant fields $\bff_2=(1,0)^{\mathsf T}$ on the triangle $T_p:=\operatorname{conv}\{(1,0),(1,1),(0,1)\}$ and $\bg_2=(1,0)^{\mathsf T}$ on the opposite-corner triangle $T_d:=\operatorname{conv}\{(-1,-1),(0,-1),(-1,0)\}$ (zero elsewhere), with disjoint supports $T_p\cap T_d=\emptyset$; one uniform refinement resolves both triangles exactly. With $\bsigma:=-\nabla u-\bff_2$ and $\br:=-\nabla z-\bg_2$ the fluxes satisfy $\nabla\cdot\bsigma=\nabla\cdot\br=0$. Since $\bff_2,\bg_2$ are discontinuous, the sources are distributional, supported on $\partial T_p$, $\partial T_d$; the dominant primal and adjoint singular structures are thus generated near $T_p$ and $T_d$ and are spatially separated. The data define
\[
\mathcal J_p(u):=-(\bg_2,\nabla u)_\Omega=-\int_{T_d}\partial_x u\,dx,
\qquad
\mathcal J_d(z):=-(\bff_2,\nabla z)_\Omega=-\int_{T_p}\partial_x z\,dx,
\]
with common value $\gamma:=\mathcal J_p(u)=(\nabla u,\nabla z)_\Omega=\mathcal J_d(z)$; a Dirichlet eigenfunction expansion on $\Omega$ yields the reference value $\gamma_{\rm ref}=-0.006340363264$.

Near $T_p$ the primal least-squares contribution is large and the adjoint one small, and vice versa near $T_d$; hence the elementwise product $\eta_{p,K}^{\rm LS}\eta_{d,K}^{\rm LS}$ can be small on \emph{every} element even though the global $\eta_p^{\rm LS},\eta_d^{\rm LS}$ are both substantial. The blind spot thus concerns the \emph{local} product used for marking, not the global one. We instead use the balanced indicator \eqref{adaptive_balanced_indicator},
$
\eta_K^2
=\frac{1}{2}\left(\frac{\eta_d^{\rm LS}}{\eta_p^{\rm LS}}\bigl(\eta_{p,K}^{\rm LS}\bigr)^2
+\frac{\eta_p^{\rm LS}}{\eta_d^{\rm LS}}\bigl(\eta_{d,K}^{\rm LS}\bigr)^2\right)$,
with the D\"orfler criterion of Section~\ref{sec:adaptive-algorithm} on a common mesh. As Figure~\ref{fig:mommer-stevenson}(a) shows, refinement is driven near the edges and vertices of both $T_p$ and $T_d$, not only where the local residuals overlap. In the convergence history, Figure~\ref{fig:mommer-stevenson}(b), the estimators $\eta_p^{\rm LS}$ and $\eta_d^{\rm LS}$ coincide and decay as $O(N^{-1/2})$, so $\eta_p^{\rm LS}\eta_d^{\rm LS}=O(N^{-1})$, and both corrected outputs $|\gamma_{\rm ref}-\gamma_{i,h}|$ exhibit a decay close to $O(N^{-1})$, with moderate preasymptotic oscillations. Neither corrected approximation stagnates despite the spatial separation of the two singular structures. This is the configuration in which purely multiplicative localization is most exposed: the elementwise product vanishes on any element where one of the two local factors vanishes, and by Cauchy--Schwarz it does not decompose the global product additively, whereas the balanced indicator satisfies the exact identity of Lemma~\ref{lem:adaptive_balanced_sum}.

\begin{figure}[htbp]
\centering
\begin{minipage}[t]{0.44\textwidth}
    \centering
    \includegraphics[width=\textwidth]{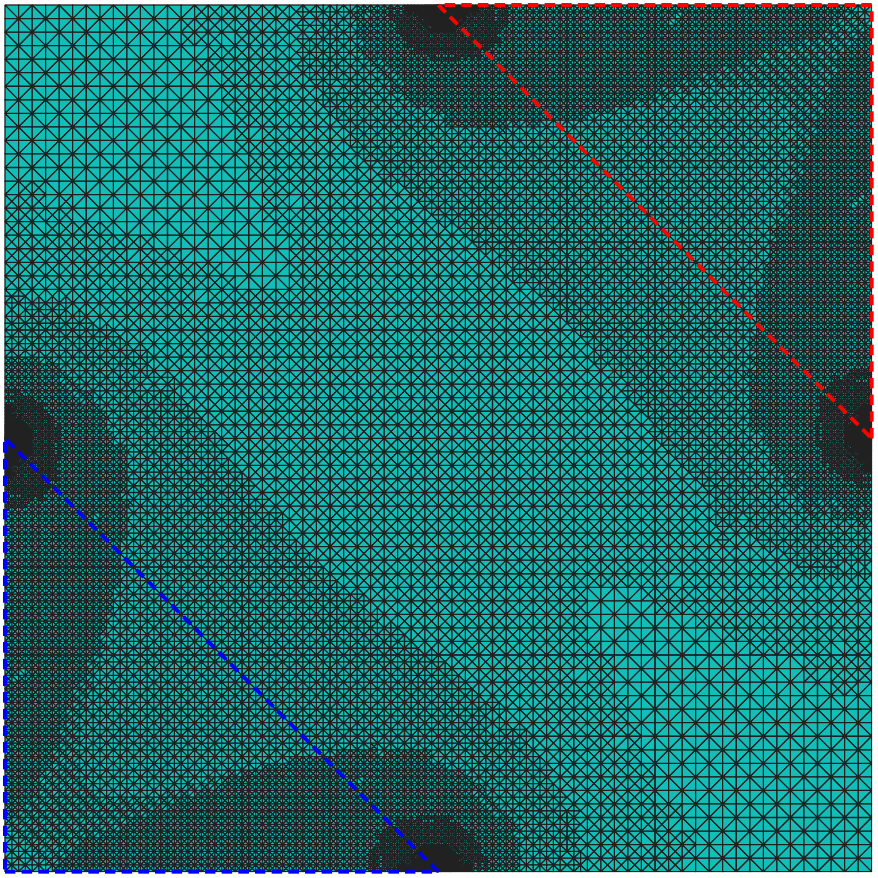}

    \smallskip{\small (a) Final adaptive mesh.}
\end{minipage}\hfill
\begin{minipage}[t]{0.52\textwidth}
    \centering
    \includegraphics[width=\textwidth]{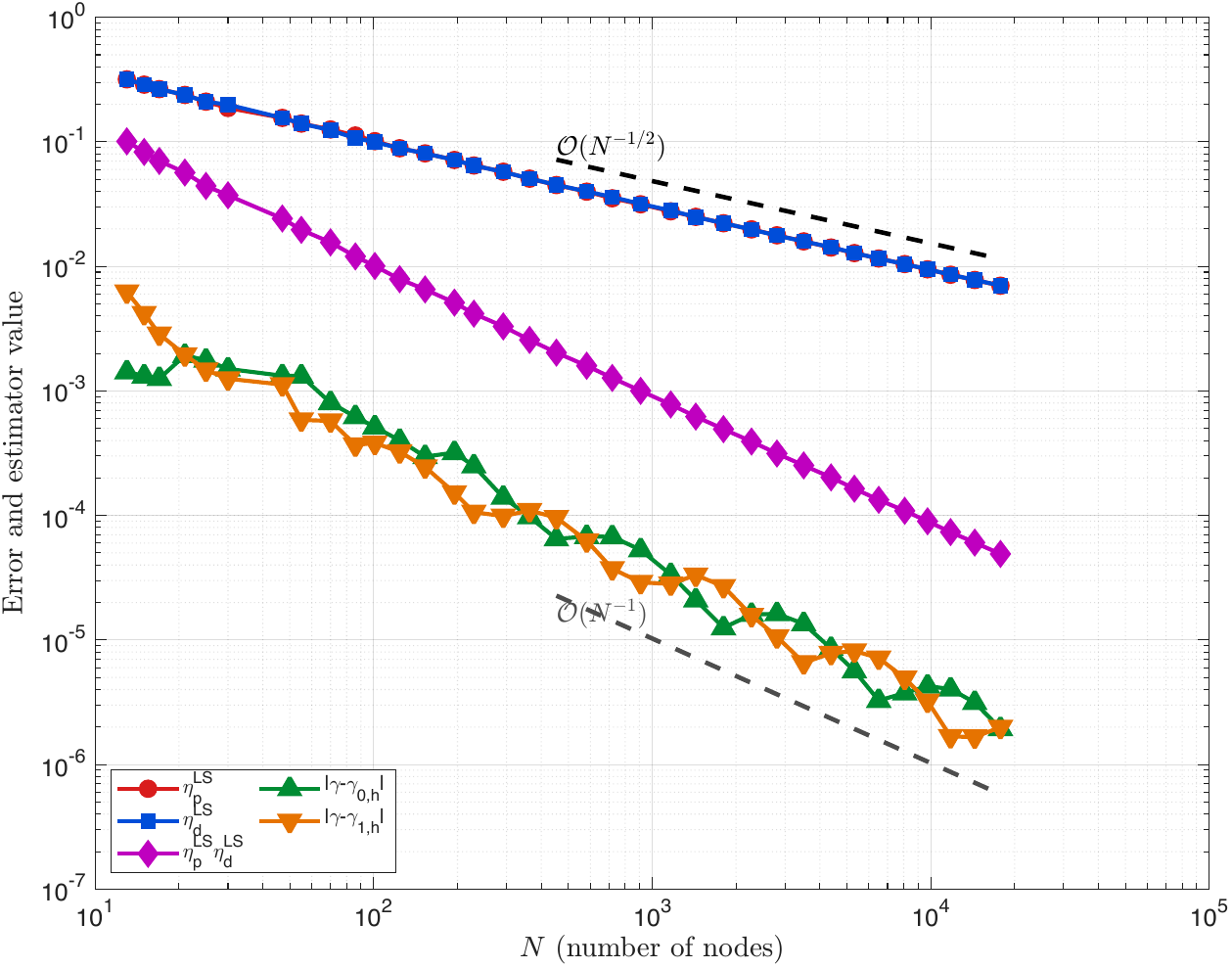}

    \smallskip{\small (b) Adaptive convergence history.}
\end{minipage}
\caption{Mommer--Stevenson benchmark. Left: final adaptive mesh from the balanced primal--adjoint marking strategy, with $T_p$ (red) and $T_d$ (blue) indicated; refinement concentrates near both active triangles. Right: adaptive convergence of the primal and adjoint least-squares estimators $\eta_p^{\rm LS}$, $\eta_d^{\rm LS}$ (which coincide here), their product $\eta_p^{\rm LS}\eta_d^{\rm LS}$, and the corrected output errors $|\gamma_{\rm ref}-\gamma_{0,h}|$, $|\gamma_{\rm ref}-\gamma_{1,h}|$.}
\label{fig:mommer-stevenson}
\end{figure}

\subsection{A convection--diffusion problem with primal and dual boundary layers}
\label{subsec:convection-diffusion}
Let \(\Omega=(0,1)^2\), \(A=0.1\,I\), \(\bb=(2,3)^{\mathsf T}\), \(c=1\), \(\Gamma_D=\partial\Omega\). We take the manufactured primal solution \(u(x,y)=(x-e^{20(x-1)})(y^2-e^{30(y-1)})\), with outflow layers at \(x=1,y=1\), and the reflected adjoint \(z(x,y):=u(1-x,1-y)\), with layers at \(x=0,y=0\). Setting \(\bsigma:=-A\nabla u\), \(\br:=-A\nabla z-\bb z\), \(\bff_2=\bg_2=0\), and taking \(f_1,g_1,\phi_p,\phi_d\) from the exact fields, the first-order systems and the quantity of interest \(\gamma:=\mathcal J_p(u)=\mathcal J_d(z)\) are as in Section~\ref{subsec:adjoint-problem}; direct integration gives \(\gamma=0.1052871614333\).

Figure~\ref{fig:boundary-layer-results} shows the final mesh and convergence history. The balanced strategy refines all four layers, most strongly near the corners where two layers meet. Both flux--potential errors decay as \(O(N^{-1/2})\), and each least-squares estimator tracks its own error, so \(\eta_p^{\rm LS}\eta_d^{\rm LS}=O(N^{-1})\); the corrected outputs \(|\gamma-\gamma_{i,h}|\) exhibit a decay close to \(O(N^{-1})\).

\begin{figure}[htbp]
\centering
\begin{minipage}[t]{0.44\textwidth}
    \centering
    \includegraphics[width=\textwidth]{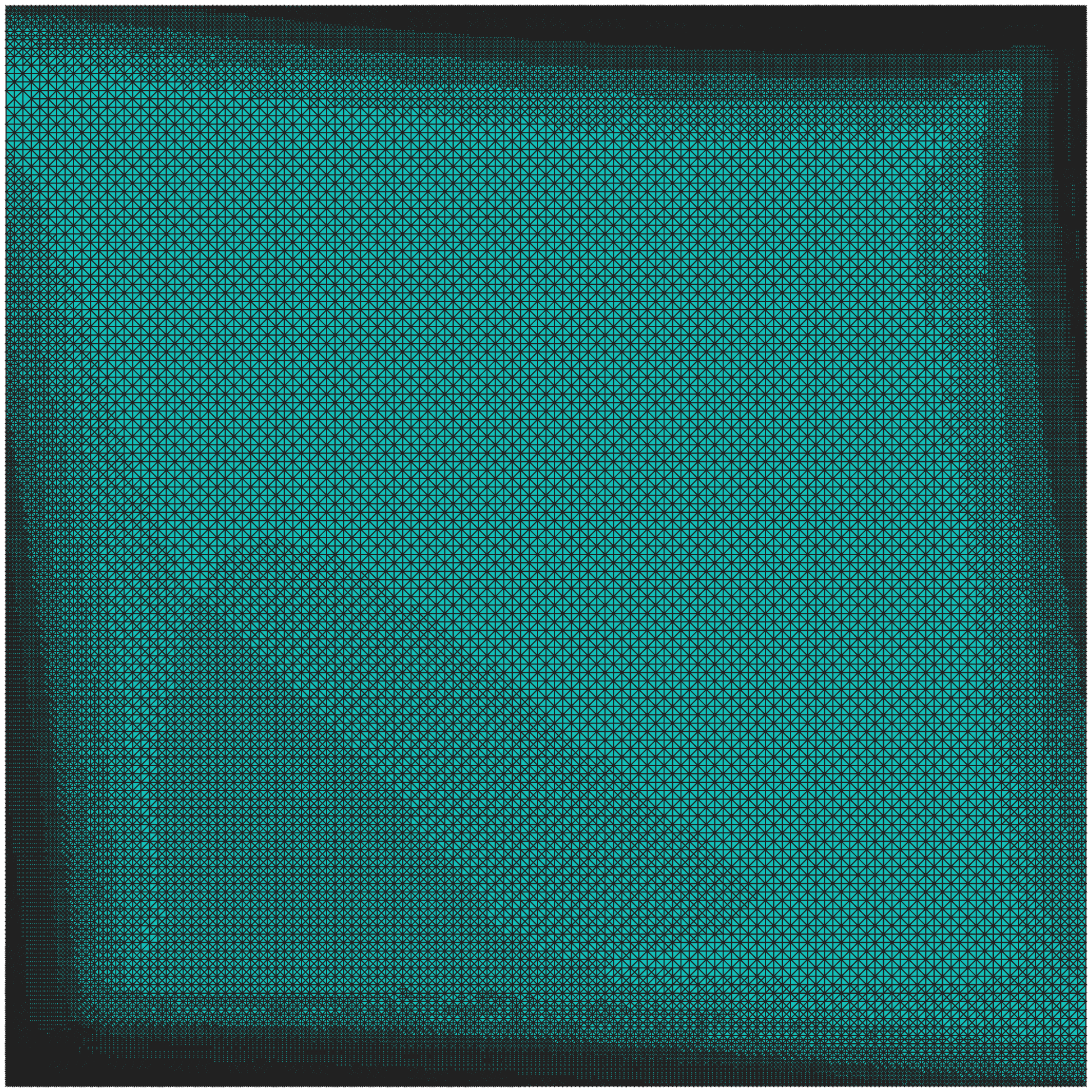}
    \smallskip{\small (a) Final adaptive mesh.}
\end{minipage}\hfill
\begin{minipage}[t]{0.52\textwidth}
    \centering
    \includegraphics[width=\textwidth]{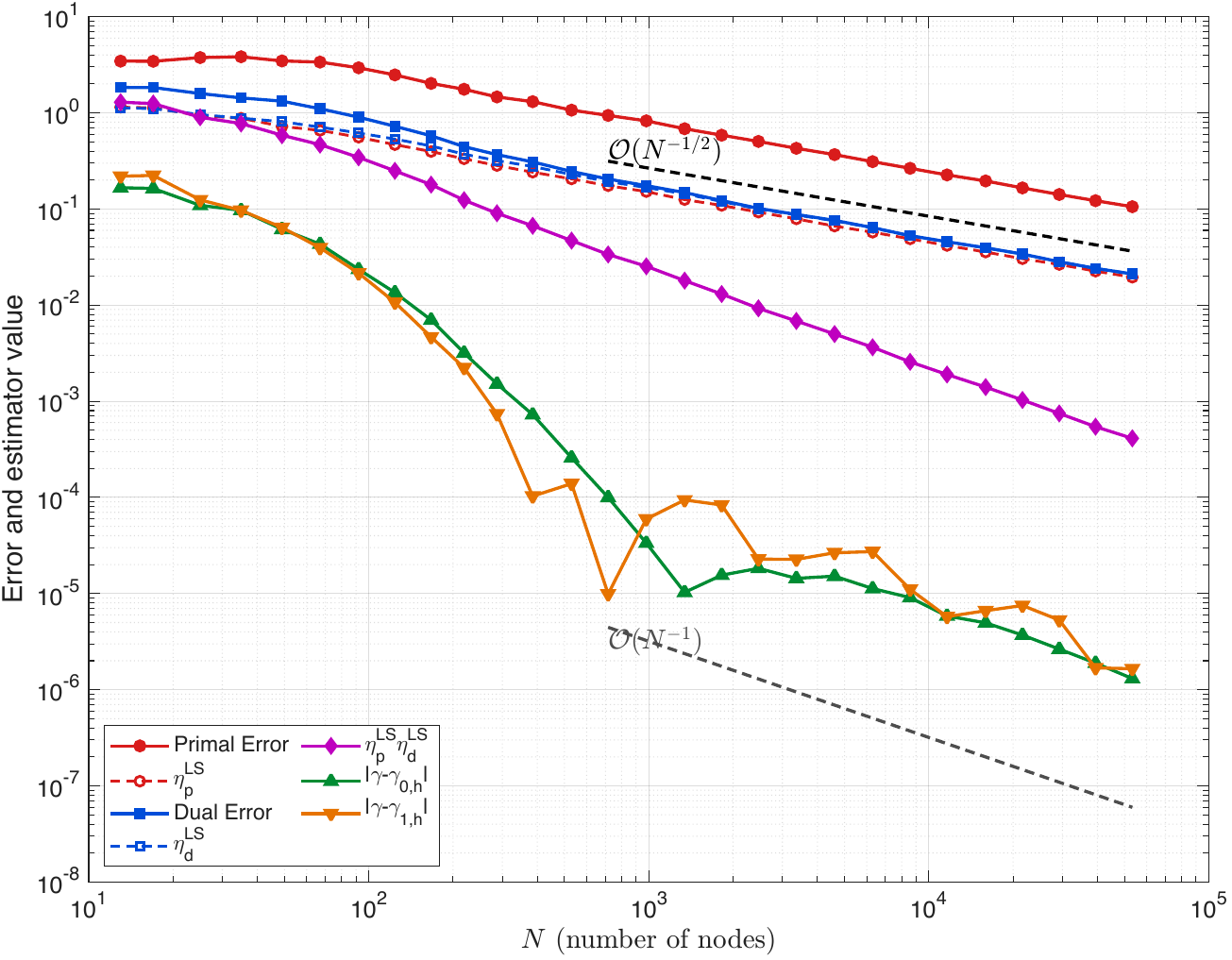}
    \smallskip{\small (b) Adaptive convergence history.}
\end{minipage}
\caption{Convection--diffusion problem with primal and adjoint boundary layers. Left: final adaptive mesh, resolving the primal layers at \(x=1,y=1\) and the adjoint layers at \(x=0,y=0\). Right: primal and adjoint errors, their least-squares estimators, the product estimator, and the corrected output errors against \(N\).}
\label{fig:boundary-layer-results}
\end{figure}

\subsection{An L-shaped problem with four quantities of interest}
\label{subsec:lshape}

Let
\[
\O=(-1,1)^2\setminus\bigl([0,1]\times[-1,0]\bigr),
\qquad
\G_N:=\{(x,1):-1<x<1\},
\qquad
\G_D:=\partial\O\setminus\overline{\G_N},
\]
and $\G_L:=\{(-1,y):-1<y<1\}\subset\G_D$. The origin is a reentrant
corner of opening angle $3\pi/2$; $(r,\theta)$ denote polar coordinates
centered there with $\theta\in(0,3\pi/2)$, so that $\theta=0$ and
$\theta=3\pi/2$ are the two edges meeting at the corner. Set
\[
u_{\rm ex}(r,\theta):=r^{2/3}\sin\Bigl(\frac{2\theta}{3}\Bigr)
\in H^{5/3-\epsilon}(\O)\setminus H^{5/3}(\O)
\qquad\text{for every }\epsilon>0 .
\]
The primal problem is
\begin{equation}\label{lshape_primal}
\begin{aligned}
-\Delta u&=0 &&\text{in }\O,\\
u&=\phi_p:=u_{\rm ex}|_{\G_D} &&\text{on }\G_D,\\
\bsigma\cdot\bn&=\psi_p:=-\nabla u_{\rm ex}\cdot\bn|_{\G_N} &&\text{on }\G_N,
\end{aligned}
\qquad
\bsigma:=-\nabla u,
\end{equation}
so $u=u_{\rm ex}$ and $\gradt\bsigma=0$. Since $u_{\rm ex}$ vanishes at
$\theta=0$ and $\theta=3\pi/2$, the datum $\phi_p$ is identically zero on
the two edges adjacent to the reentrant corner and smooth elsewhere on
$\G_D$. A single primal singularity
is probed below by four physical outputs, each inducing its own adjoint,
so the marking strategy of Section~\ref{sec:adaptive-algorithm} must adapt
to a different adjoint structure in each case.

\subsubsection{Four quantities of interest and their adjoints}
\label{subsec:lshape-qoi}
The four outputs exercise, in turn, the Neumann-trace, Dirichlet-flux,
volume and gradient terms of \eqref{cJp}:
\[
J_1(u)=\int_{\G_N}u\,ds,\qquad
J_2(\bsigma)=\int_{\G_L}(y+1)\,\bsigma\cdot\bn\,ds,\qquad
J_3(u)=\int_{\omega_0}u\,dx,\qquad
J_4(\bsigma)=\int_{T_0}\sigma_1\,dx,
\]
where
\[
\omega_0:=\Bigl(-1,-\tfrac12\Bigr)^2,
\qquad
T_0:=\operatorname{conv}\Bigl\{(-1,-1),(-1,0),\bigl(-\tfrac12,-\tfrac12\bigr)\Bigr\},
\]
whose interiors lie in $\O$, whose closures meet $\G_D$, and which are
aligned with the initial mesh, so that $\chi_{\omega_0}$ and $\bg_2$ are piecewise
constant on every mesh used; $\chi_\omega$ denotes the characteristic
function of $\omega$. The corresponding data in
\eqref{cJp} and adjoint problems \eqref{eq_dual} are as follows; in each
case $\br_i:=-(\nabla z_i+\bg_2)$.

\medskip
\noindent\emph{(i) Neumann trace.} $g_1=0$, $\bg_2=\mathbf 0$,
$\psi_d=-1$, $\phi_d=0$:
\[
-\Delta z_1=0\ \text{in }\O,\qquad
z_1=0\ \text{on }\G_D,\qquad
\br_1\cdot\bn=-1\ \text{on }\G_N .
\]
Besides the reentrant-corner singularity, $z_1$ has weaker mixed-boundary
singularities at the Neumann endpoints $(\pm1,1)$.

\medskip
\noindent\emph{(ii) Weighted Dirichlet flux.} $g_1=0$, $\bg_2=\mathbf 0$,
$\psi_d=0$, $\phi_d=(y+1)\chi_{\G_L}$:
\[
-\Delta z_2=0\ \text{in }\O,\qquad
z_2=(y+1)\chi_{\G_L}\ \text{on }\G_D,\qquad
\br_2\cdot\bn=0\ \text{on }\G_N .
\]
The weight of the flux measurement has become the essential datum of the
adjoint, as predicted by \eqref{dual_essential_data}. In addition to the
corner singularity, $z_2$ has a mixed-boundary singularity at $(-1,1)$,
where the datum meets the homogeneous Neumann condition without
first-order compatibility.

\medskip
\noindent\emph{(iii) Localized volume.} $g_1=\chi_{\omega_0}$,
$\bg_2=\mathbf 0$, $\psi_d=0$, $\phi_d=0$:
\[
-\Delta z_3=\chi_{\omega_0}\ \text{in }\O,\qquad
z_3=0\ \text{on }\G_D,\qquad
\br_3\cdot\bn=0\ \text{on }\G_N .
\]
Here $z_3$ retains the corner singularity and has reduced regularity
across $\partial\omega_0$, where the source jumps.

\medskip
\noindent\emph{(iv) Localized flux.} $g_1=0$,
$\bg_2=(\chi_{T_0},0)^{\mathsf T}$, $\psi_d=0$, $\phi_d=0$, so that
$\cJ_p(u)=-(\bg_2,\nabla u)_\O=(\bg_2,\bsigma)_\O=J_4(\bsigma)$:
\[
\br_4+\nabla z_4=-\bg_2,\quad \gradt\br_4=0\ \ \text{in }\O,
\qquad
z_4=0\ \text{on }\G_D,
\qquad
\br_4\cdot\bn=0\ \text{on }\G_N,
\]
equivalently $-\Delta z_4=\partial_x\chi_{T_0}$ in $\mathcal D'(\O)$.
Thus $z_4$ is harmonic in $T_0$ and in $\O\setminus\overline{T_0}$, while
the jump of $\bg_2$ across $\partial T_0$ acts as a distributional source
on $\partial T_0$, producing singular behavior at the vertices of $T_0$.

\medskip
Each output is evaluated exactly at $u=u_{\rm ex}$ by elementary
integration; the resulting values are collected in
Table~\ref{tab:lshape-qoi}.

\begin{table}[htbp]
\centering
\small
\begin{tabular}{lll}
\toprule
QoI & Type & Exact value $\gamma$\\
\midrule
$J_1$ & Neumann trace           & $\phantom{-}1.78860312648842$\\
$J_2$ & Weighted Dirichlet flux & $-1.03265049655150$\\
$J_3$ & Localized volume        & $\phantom{-}0.130002152935573$\\
$J_4$ & Localized flux          & $\phantom{-}0.157712193795526$\\
\bottomrule
\end{tabular}
\caption{L-shaped example: the four quantities of interest and their
exact values.}
\label{tab:lshape-qoi}
\end{table}

\subsubsection{Adaptive results}
\label{subsec:lshape-results}
For all four outputs the primal error and both estimators decay as
$O(N^{-1/2})$, so $\eta_p^{\mathrm{LS}}\eta_d^{\mathrm{LS}}=O(N^{-1})$,
and both corrected outputs exhibit a decay close to this product rate
(Figure~\ref{fig:lshape-convergence}), consistent with
Theorem~\ref{thm:gamma1_error} and Corollary~\ref{cor:qoi_rate}; the
moderate preasymptotic oscillations come from cancellation among the
signed terms of the corrected functionals. The flux--potential correction
$\gamma_{1,h}$ is markedly more accurate than $\gamma_{0,h}$ for the flux
output $J_4$ and comparable for $J_1$--$J_3$. Driven by the balanced
indicator \eqref{adaptive_balanced_indicator}, refinement resolves the
reentrant-corner singularity common to all four problems together with the
output-dependent adjoint features visible in
Figure~\ref{fig:lshape-duals}: the Neumann endpoints $(\pm1,1)$ for
$J_1$, the mixed-boundary point $(-1,1)$ for $J_2$, the interface
$\partial\omega_0$ for $J_3$, and the vertices of $T_0$ for $J_4$, all of
which a primal-only strategy would miss.

\begin{figure}[htbp]
\centering
\begin{minipage}[t]{0.24\textwidth}\centering
  \includegraphics[width=\textwidth]{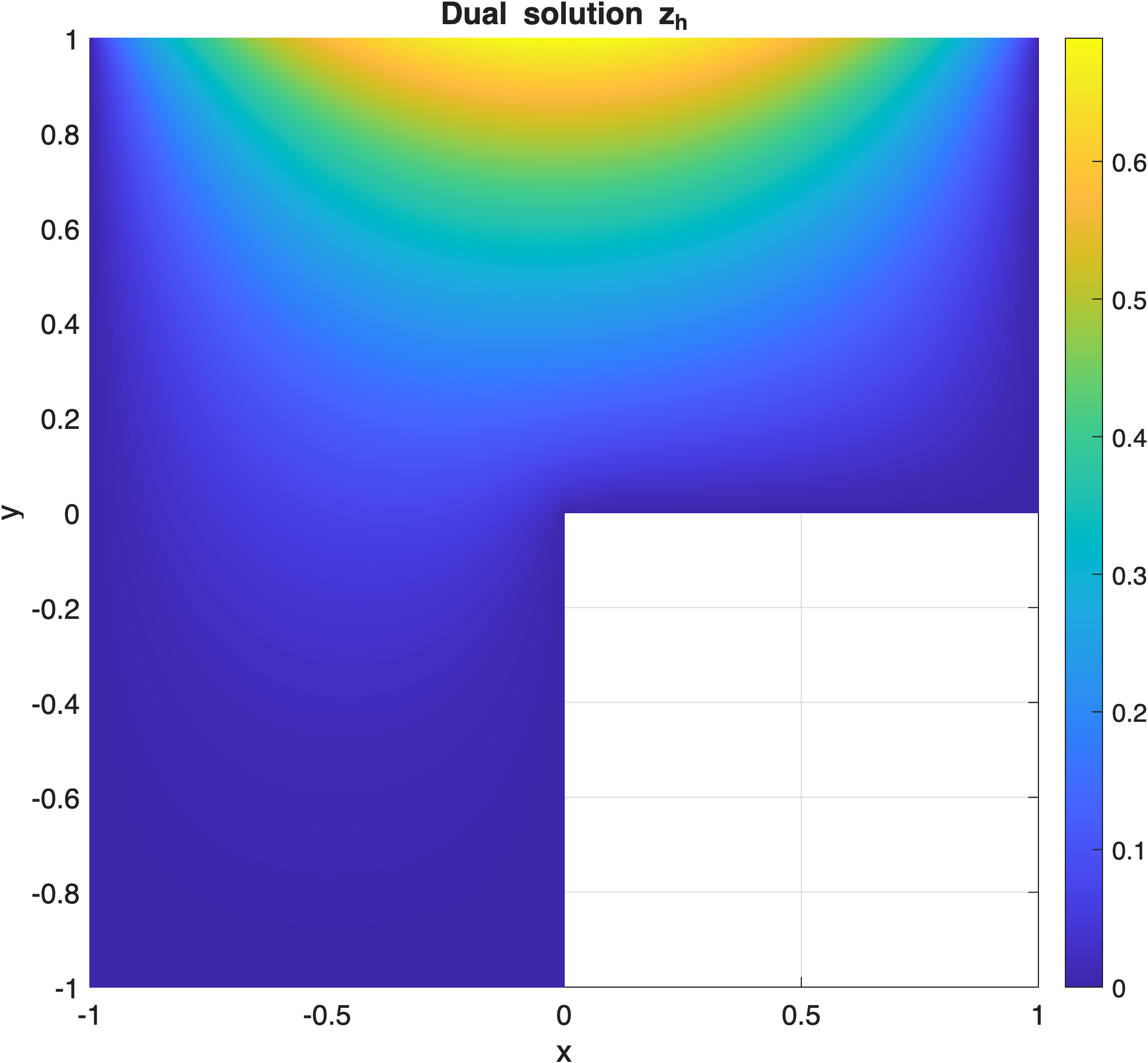}\\
  \smallskip{\small (a) $J_1$}\end{minipage}\hfill
\begin{minipage}[t]{0.24\textwidth}\centering
  \includegraphics[width=\textwidth]{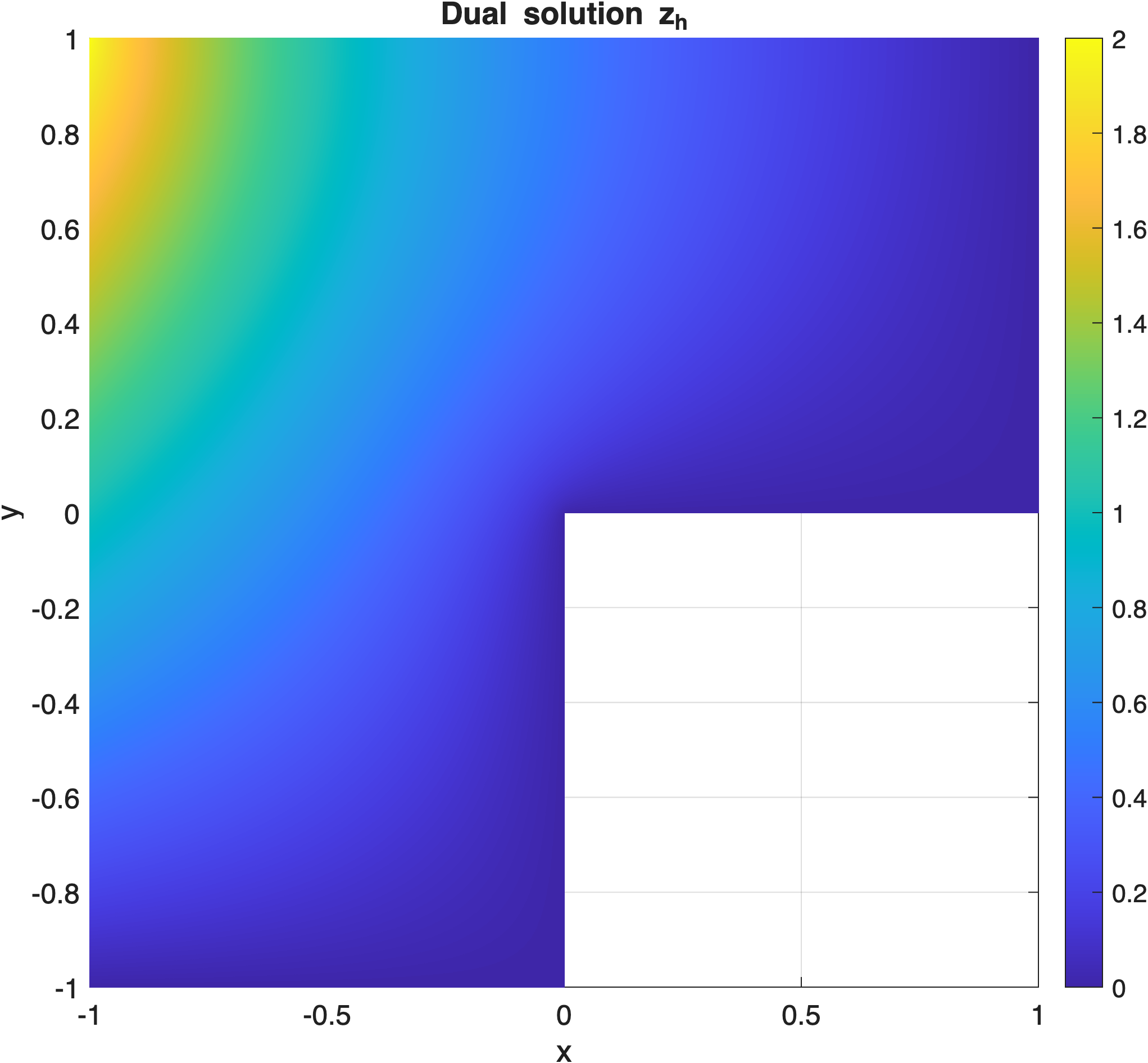}\\
  \smallskip{\small (b) $J_2$}\end{minipage}\hfill
\begin{minipage}[t]{0.24\textwidth}\centering
  \includegraphics[width=\textwidth]{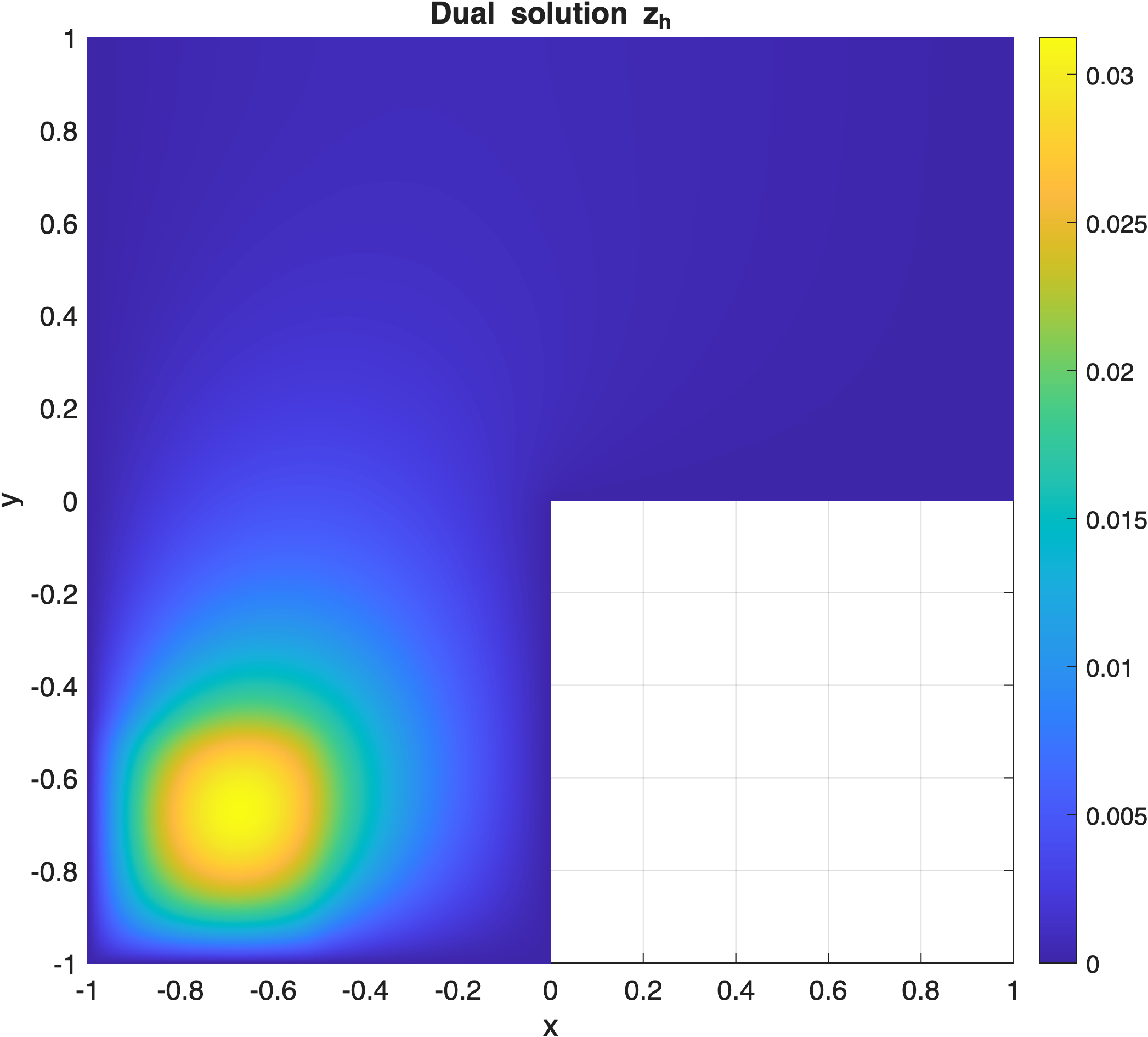}\\
  \smallskip{\small (c) $J_3$}\end{minipage}\hfill
\begin{minipage}[t]{0.24\textwidth}\centering
  \includegraphics[width=\textwidth]{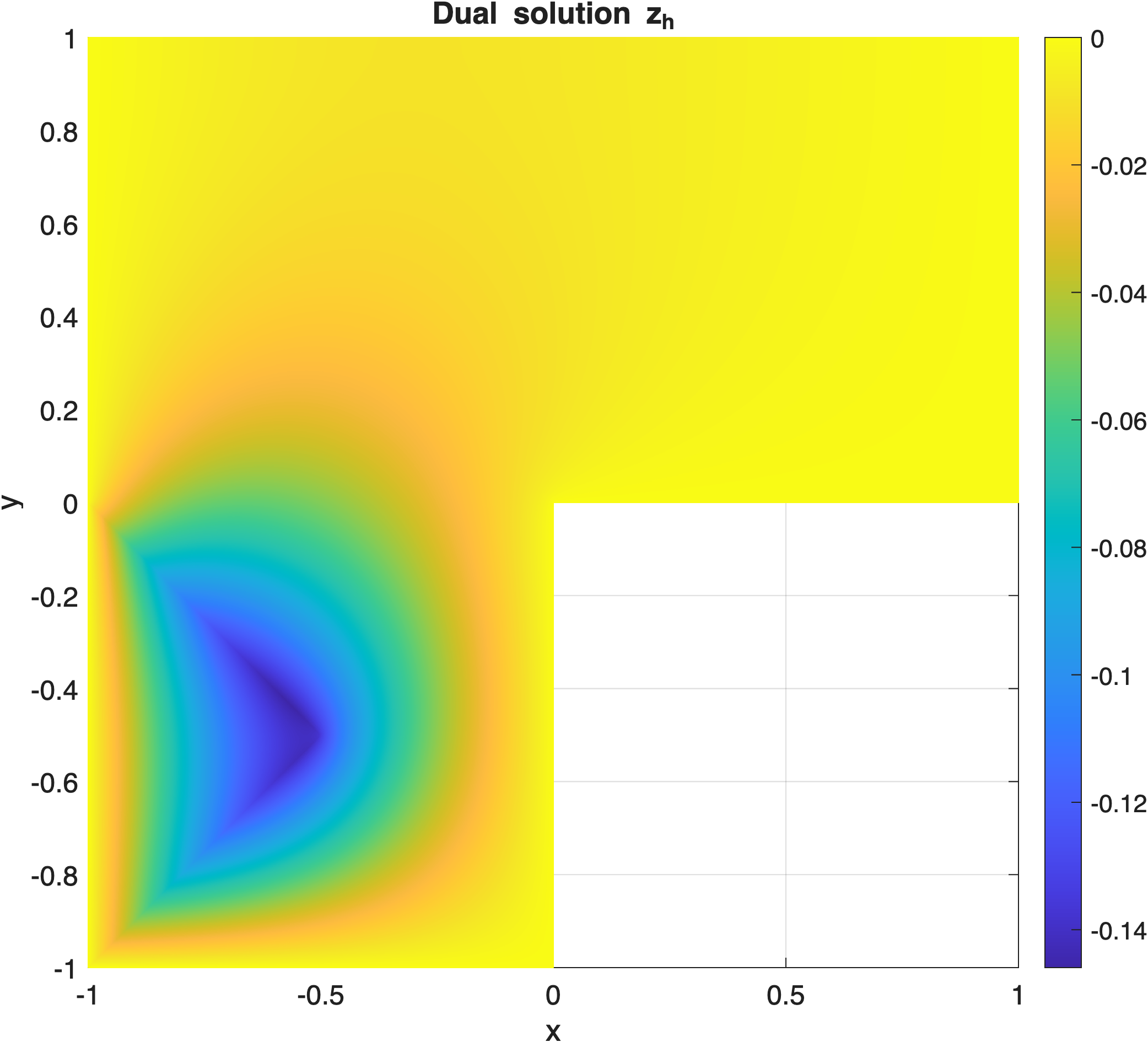}\\
  \smallskip{\small (d) $J_4$}\end{minipage}
\caption{L-shaped example: computed adjoint solutions $z_h$ for the four
quantities of interest, each induced by a different physical output.}
\label{fig:lshape-duals}
\end{figure}

\begin{figure}[htbp]
\centering
\begin{minipage}[t]{0.45\textwidth}\centering
  \includegraphics[width=\textwidth]{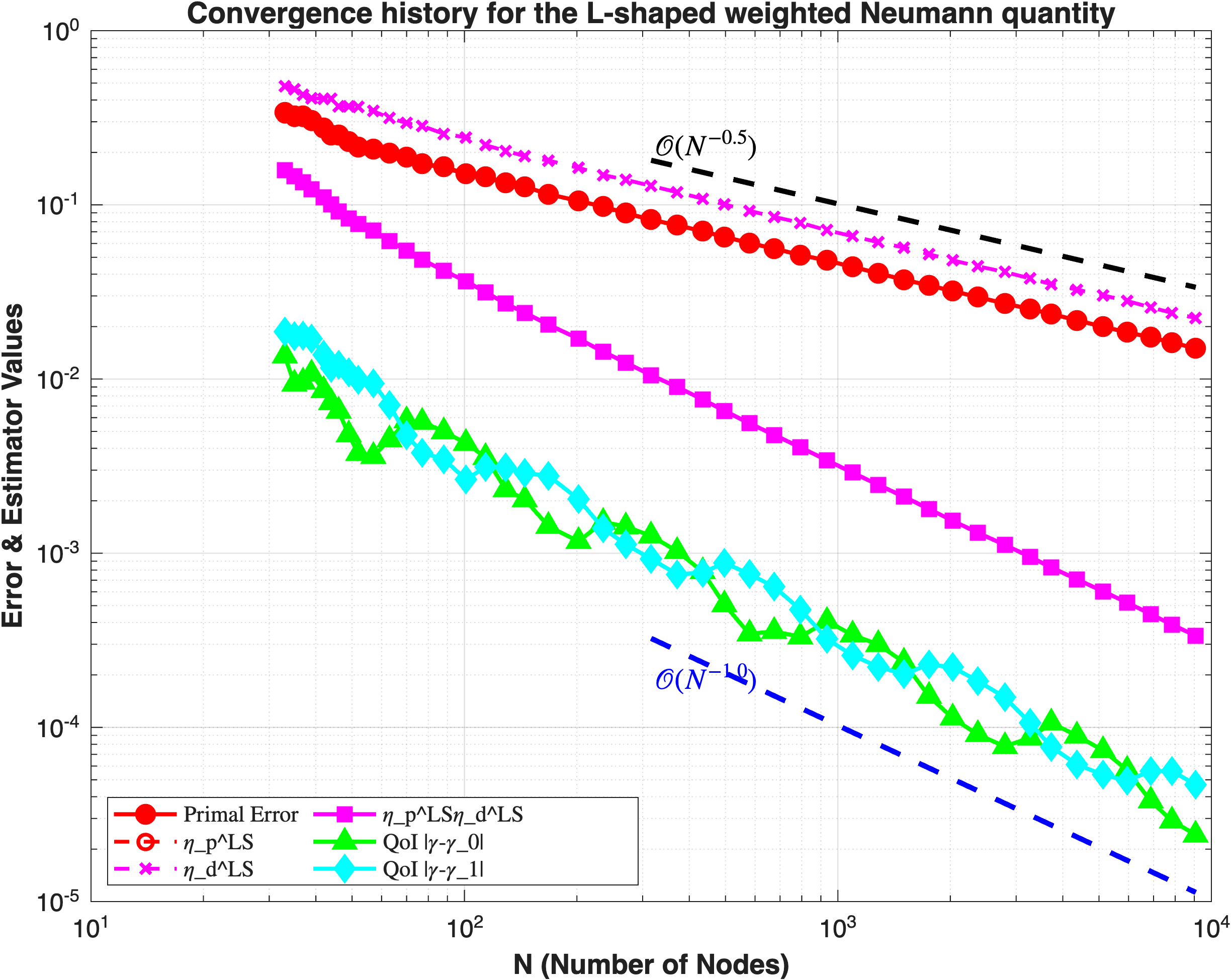}\\
  \smallskip{\small (a) $J_1$}\end{minipage}\hfill
\begin{minipage}[t]{0.45\textwidth}\centering
  \includegraphics[width=\textwidth]{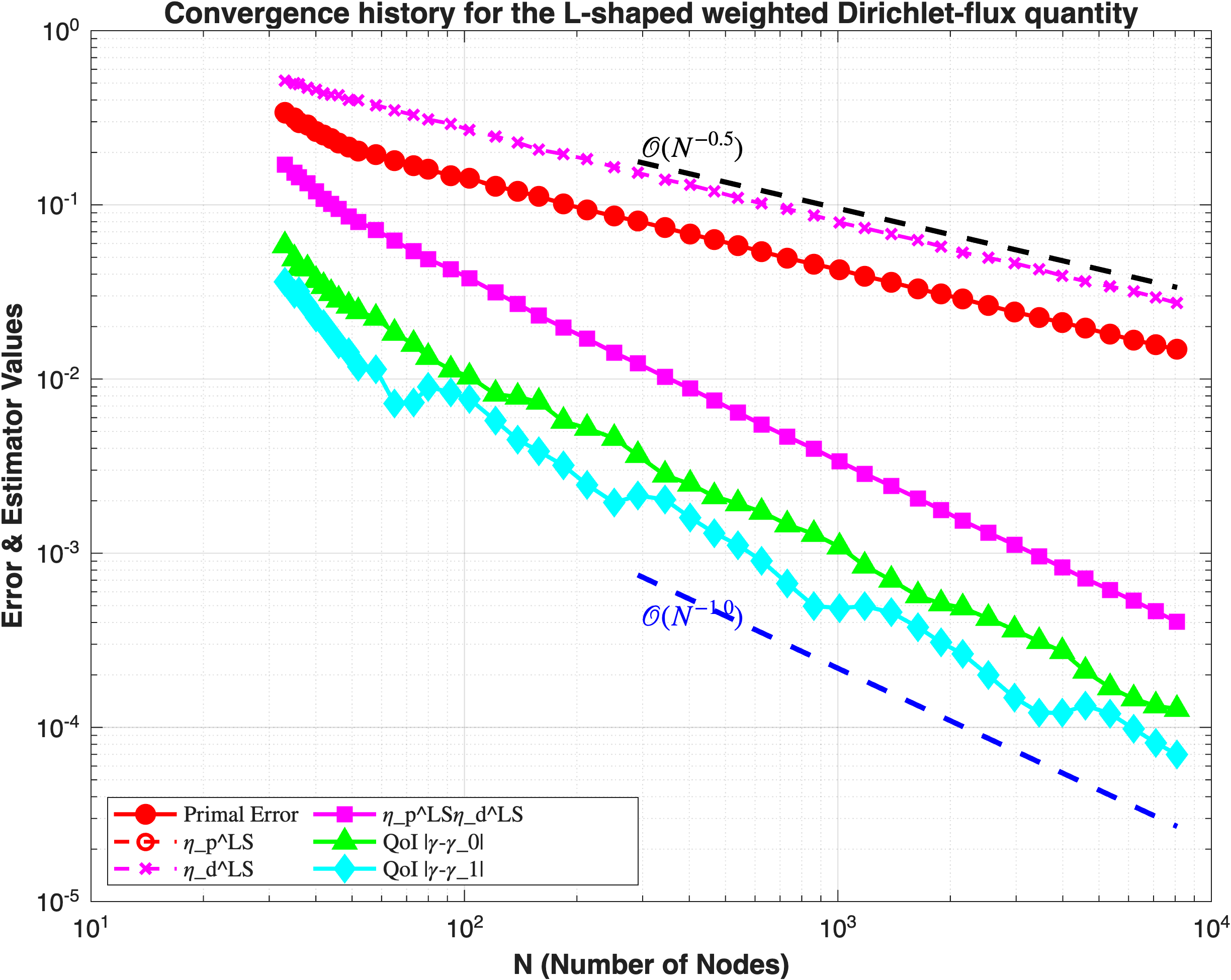}\\
  \smallskip{\small (b) $J_2$}\end{minipage}

\medskip
\begin{minipage}[t]{0.45\textwidth}\centering
  \includegraphics[width=\textwidth]{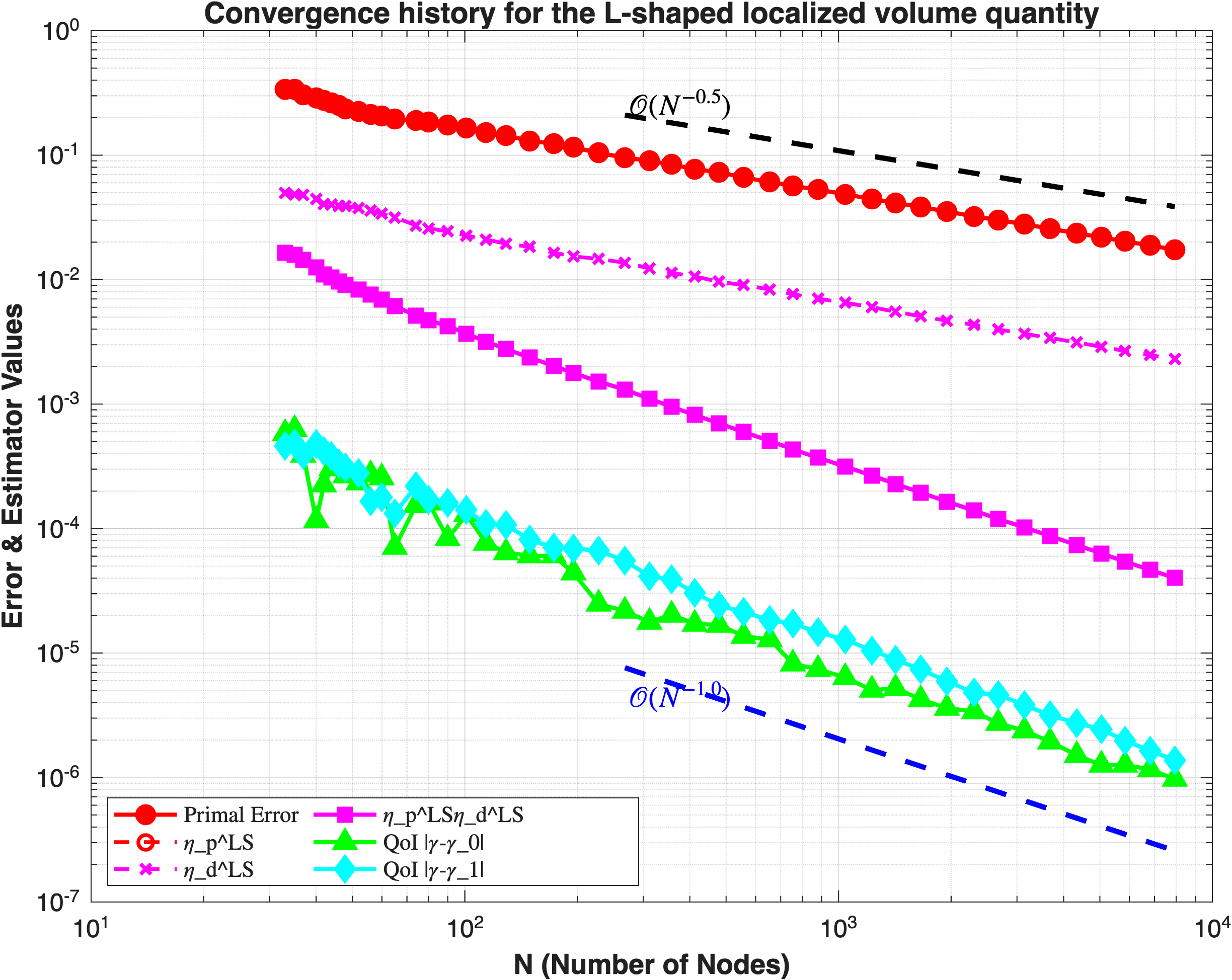}\\
  \smallskip{\small (c) $J_3$}\end{minipage}\hfill
\begin{minipage}[t]{0.45\textwidth}\centering
  \includegraphics[width=\textwidth]{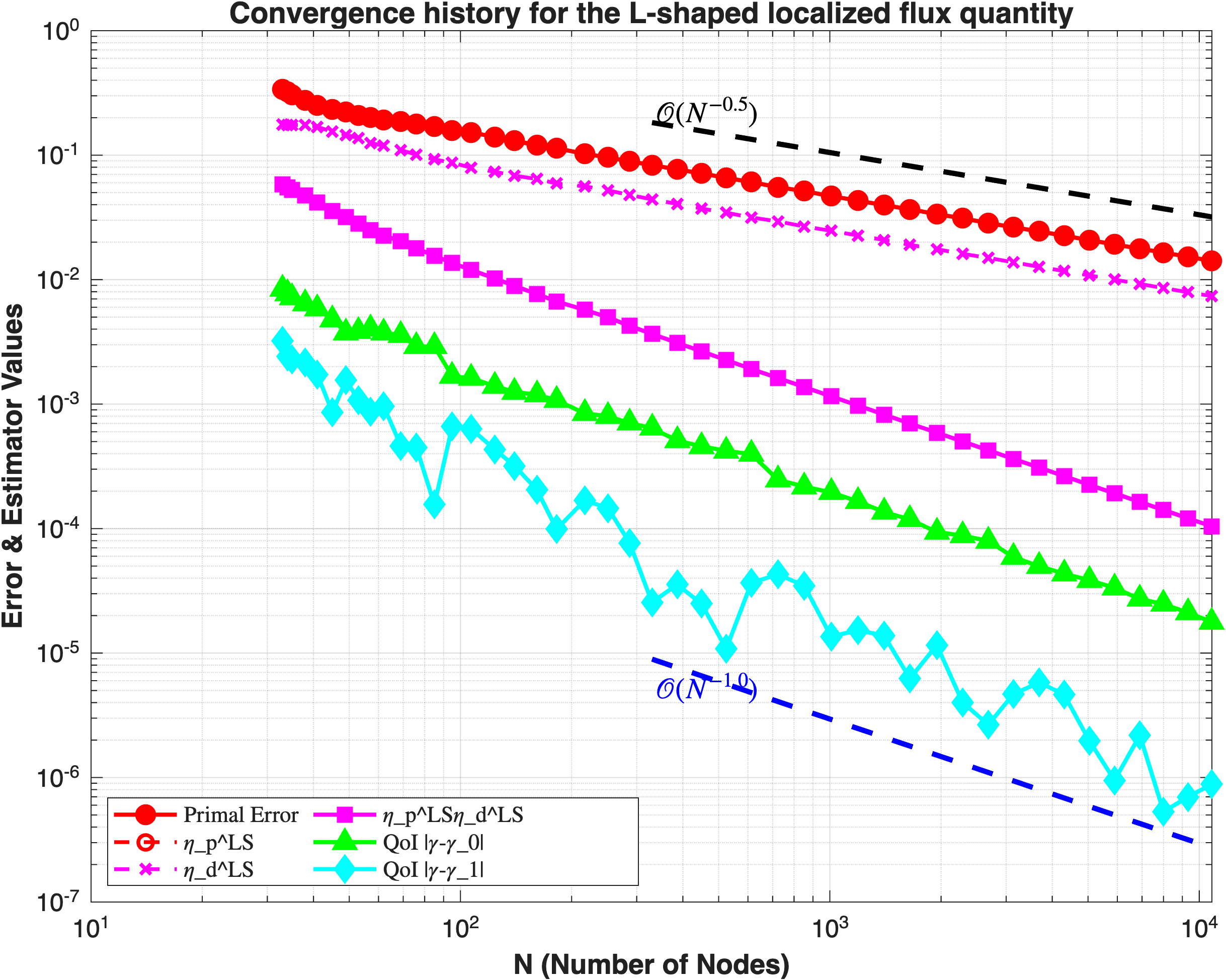}\\
  \smallskip{\small (d) $J_4$}\end{minipage}
\caption{L-shaped example: adaptive convergence histories for the four
quantities of interest. Each panel shows the primal error, the primal and
adjoint least-squares estimators, their product, and the corrected output
errors $|\gamma-\gamma_{0,h}|$, $|\gamma-\gamma_{1,h}|$ against $N$; dashed
lines mark $O(N^{-1/2})$ and $O(N^{-1})$.}
\label{fig:lshape-convergence}
\end{figure}

\section{Concluding Remarks}
\label{sec:conclusions}

We have developed goal-oriented error estimation for first-order system
least-squares methods around the physical PDE adjoint. The corrected
values $\gamma_0$ and $\gamma_1$ satisfy exact product-type error
identities that follow from the continuous primal and adjoint equations
alone, with no Galerkin orthogonality; because the physical adjoint is
itself a boundary value problem with a first-order flux system, it
carries a native least-squares functional, so a built-in estimator is
available for the dual problem as well as the primal one.

Since these identities are independent of the discretization, the
framework is not confined to least-squares approximations. They also
suggest an extension to DPG \cite{DPG:acta2025,CDG:14} and other
minimum-residual schemes, provided that the corresponding residual
estimators control the primal and PDE-adjoint norms entering the
corrected identity. The framework applies in non-intrusive
settings: given conforming potential approximations produced by another
discretization, an $H(\divvr;\O)$-conforming flux recovery
\cite{CZ:10b,LZYZ:24} matching the Neumann data supplies
admissible pairs for $\gamma_1$, and the least-squares functionals
\eqref{LS_functional_primal}--\eqref{LS_functional_dual} remain reliable
and efficient estimators of the resulting flux--potential errors by
Theorems~\ref{thm:primal_LS_estimator} and~\ref{thm:dual_LS_estimator},
so the goal-oriented bounds apply verbatim to an approximation the
least-squares method did not produce. Augmented mixed methods, in which
least-squares terms are added to the mixed formulation so that the
potential is sought in $H^1(\O)$ alongside an $H(\divvr;\O)$-conforming
flux \cite{LZ:24}, supply such pairs directly, with no recovery step:
the potential is conforming and the flux is admissible in the sense of
Remark~\ref{rem:gamma1_hypotheses}, so both corrected values and the
least-squares functional estimators are available unchanged. Neural network approximations fit the same pattern: deep least-squares methods \cite{CCLL:20} minimize the functionals \eqref{LS_functional_primal}--\eqref{LS_functional_dual} over classes of network functions, and since the identities of Section~\ref{sec:qoi-error-identities} and Theorems~\ref{thm:primal_LS_estimator} and~\ref{thm:dual_LS_estimator} concern arbitrary admissible pairs rather than a linear approximation space, they apply to such approximations unchanged, provided the essential boundary conditions are imposed exactly.

\appendix
\section{Expansion of \texorpdfstring{$\gamma-\widetilde\gamma_1$}{gamma - gamma1-tilde}}
\label{app:expansion}

We prove the expansion \eqref{tilde_gamma1_error_expansion}. Write
$E_p=\bsigma-\btau$, $e_p=u-w$, $E_d=\br-\brho$, $e_d=z-v$, so that
$\btau=\bsigma-E_p$, $w=u-e_p$, $\brho=\br-E_d$, $v=z-e_d$. The map
$\widetilde\gamma_1$ is affine in $(\btau,w)$ and in $(\brho,v)$, and by the
exact-pair identities $\widetilde m(\bsigma,u)=m(u)$,
$\widetilde\ell(\br,z)=\ell(z)$,
$\widetilde a((\bsigma,u),(\br,z))=a(u,z)$ together with
\eqref{basic_gamma_representation}, $\widetilde\gamma_1(u,\bsigma,z,\br)=\gamma$.
Hence
\begin{equation}\label{app:split}
\gamma-\widetilde\gamma_1
=\bigl[\widetilde m(\bsigma,u)-\widetilde m(\btau,w)\bigr]
+\bigl[\widetilde\ell(\br,z)-\widetilde\ell(\brho,v)\bigr]
-\bigl[\widetilde a((\bsigma,u),(\br,z))-\widetilde a((\btau,w),(\brho,v))\bigr].
\end{equation}

\emph{The three differences.} From \eqref{def_tilde_m}--\eqref{def_tilde_a},
using bilinearity and cancelling the constant parts,
\begin{align}
\widetilde m(\bsigma,u)-\widetilde m(\btau,w)
&=(g_1,e_p)+(A^{-1}E_p,\bg_2)-\langle\psi_d,e_p\rangle_{\G_N},
\label{app:Im}\\
\widetilde\ell(\br,z)-\widetilde\ell(\brho,v)
&=(f_1,e_d)+\bigl(\bff_2,A^{-1}(E_d+\bb e_d)\bigr)-\langle\psi_p,e_d\rangle_{\G_N},
\label{app:Il}\\
\widetilde a((\bsigma,u),(\br,z))-\widetilde a((\btau,w),(\brho,v))
&=\bigl(A^{-1}(\bsigma+\bff_2),E_d\bigr)
 +\bigl(A^{-1}E_p,\br+\bg_2\bigr)
 -(A^{-1}E_p,E_d)\nonumber\\
&\quad+(cu,e_d)+(ce_p,z)-(ce_p,e_d).
\label{app:Ia}
\end{align}

\emph{Volume and Neumann data.} Since $e_p,e_d\in H_D^1(\O)$, the adjoint and
primal weak equations \eqref{weak_dual}, \eqref{weak} give
\[
(g_1,e_p)-\langle\psi_d,e_p\rangle_{\G_N}=a(e_p,z)+(\bg_2,\nabla e_p),
\qquad
(f_1,e_d)-\langle\psi_p,e_d\rangle_{\G_N}=a(u,e_d)+(\bff_2,\nabla e_d).
\]
Using the exact constitutive relations $A\nabla u=-(\bsigma+\bff_2)$ and
$A\nabla z=-(\br+\bb z+\bg_2)$ (pointwise), one finds
\[
a(e_p,z)=-(\nabla e_p,\br+\bg_2)+(ce_p,z),
\qquad
a(u,e_d)=-(\bsigma+\bff_2,\nabla e_d)+(\bb\cdot\nabla u,e_d)+(cu,e_d),
\]
where the $\bb z$ contribution cancels in the first identity. Substituting
these into \eqref{app:Im}--\eqref{app:Il}, the $\bg_2$- and $\bff_2$-gradient
terms cancel pairwise.

\emph{Reaction terms.} The reaction contributions $(ce_p,z)$ and
$(cu,e_d)$ just produced cancel against the corresponding terms of
\eqref{app:Ia}, which enter \eqref{app:split} with the opposite sign,
leaving
\[
(ce_p,z)+(cu,e_d)-\bigl[(cu,e_d)+(ce_p,z)-(ce_p,e_d)\bigr]=(ce_p,e_d).
\]

\emph{Collecting the flux terms.} Introduce the computable constitutive
residuals
\[
R_p^h:=\btau+A\nabla w+\bff_2,
\qquad
R_d^h:=\brho+\bg_2+A\nabla v+\bb v,
\]
which satisfy $R_p^h=-E_p-A\nabla e_p$ and $R_d^h=-E_d-A\nabla e_d-\bb e_d$.
Expressing $\nabla e_p,\nabla e_d$ through these residuals and substituting the
constitutive relations for $\bsigma,\br$ in the remaining terms, the flux
contributions collect to
\[
(A^{-1}E_p,E_d)+\bigl(R_p^h,A^{-1}E_d\bigr)+\bigl(A^{-1}E_p,R_d^h\bigr)
+\bigl(R_p^h,A^{-1}\brho\bigr)+\bigl(A^{-1}\btau,R_d^h\bigr).
\]
Together with the reaction term $(ce_p,e_d)$, this is exactly
\eqref{tilde_gamma1_error_expansion}. Folding the last two terms into the
definition \eqref{def_gamma1} of $\gamma_1$ removes
$(R_p^h,A^{-1}\brho)+(A^{-1}\btau,R_d^h)$ and yields the identity
\eqref{second_error_identity}. \qed

\section{Reference values for the numerical experiments}
\label{app:reference-values}

We record how the reference values $\gamma_{\rm ref}$ used in
Section~\ref{sec:numerical} are obtained, so that the reported output
errors are independently reproducible.

\subsection*{B.1 The smooth interface problem}

Both representations
\[
\cJ_p(u)=(g_1,u)_\O-(\bg_2,\nabla u)_\O-\langle\psi_d,u\rangle_{\G_N}
+\langle\bsigma\cdot\bn,\phi_d\rangle_{\G_D},
\qquad
\cJ_d(z)=(f_1,z)_\O-(\bff_2,\nabla z)_\O-\langle\psi_p,z\rangle_{\G_N}
+\langle\br\cdot\bn,\phi_p\rangle_{\G_D},
\]
are evaluated at the exact fields of
Section~\ref{subsec:smooth-interface-example}. Since $A$ and the
corrections $\bff_2,\bg_2$ are piecewise smooth with jumps across
$x+y=\pm1$, each volume integral is split into the three subdomains
$\O_0,\O_1,\O_2$ and each boundary integral at the points where
$\G_D,\G_N$ meet an interface; on each piece the integrand is analytic and
a tensor-product Gauss--Legendre rule is used. The two representations
then agree to
\[
\gamma_{\rm ref}=\cJ_p(u)=\cJ_d(z)=5.058351722069 ,
\]
which also provides a consistency check on the volume terms, the mixed
boundary terms, the interface corrections, and the primal--adjoint sign
conventions. An equivalent evaluation through
$\gamma=m(u)+\ell(z)-a(u,z)$, which avoids the Dirichlet-boundary flux
pairing altogether, reproduces the same value.

\subsection*{B.2 The benchmark with separated singularities}

On $\O=(-1,1)^2$ the $L^2$-orthonormal Dirichlet eigenfunctions are
\[
\phi_{nm}(x,y)=\sin\Bigl(\frac{n\pi(x+1)}2\Bigr)
               \sin\Bigl(\frac{m\pi(y+1)}2\Bigr),
\qquad
-\Delta\phi_{nm}=\lambda_{nm}\phi_{nm},
\qquad
\lambda_{nm}=\frac{\pi^2}{4}(n^2+m^2).
\]
Put $F_{nm}:=\int_{T_p}\partial_x\phi_{nm}\,dx$ and
$G_{nm}:=\int_{T_d}\partial_x\phi_{nm}\,dx$. Testing the two weak
equations with $\phi_{nm}$ gives
$\lambda_{nm}u_{nm}=-F_{nm}$ and $\lambda_{nm}z_{nm}=-G_{nm}$ for the
expansion coefficients of $u$ and $z$, whence
\[
\gamma=(\nabla u,\nabla z)_\O
=\sum_{n,m\ge1}\lambda_{nm}u_{nm}z_{nm}
=\sum_{n,m\ge1}\frac{F_{nm}G_{nm}}{\lambda_{nm}} .
\]
The two active triangles are related by the point reflection $T_d=-T_p$,
and $\phi_{nm}(-x,-y)=(-1)^{n+m}\phi_{nm}(x,y)$ gives
$G_{nm}=(-1)^{n+m+1}F_{nm}$. Hence
\begin{equation}\label{app:ms_series}
\gamma=\frac{4}{\pi^2}\sum_{n=1}^{\infty}\sum_{m=1}^{\infty}
(-1)^{n+m+1}\,\frac{F_{nm}^2}{n^2+m^2}.
\end{equation}
Since $\phi_{nm}$ vanishes on the outer edges $x=1$ and $y=1$ of $T_p$,
the divergence theorem applied to the constant field $(1,0)^{\mathsf T}$
reduces $F_{nm}$ to a single edge integral, and a direct trigonometric
evaluation gives
\[
F_{nm}=-\frac{2(-1)^m}{\pi(n^2-m^2)}
\Bigl[(-1)^n n\sin\Bigl(\frac{m\pi}2\Bigr)
      +m\sin\Bigl(\frac{n\pi}2\Bigr)\Bigr]
\quad (n\neq m),
\qquad
F_{nn}=
\begin{cases}
\tfrac12(-1)^{n/2}, & n\ \text{even},\\[1mm]
-\tfrac1{n\pi}\sin\bigl(\tfrac{n\pi}2\bigr), & n\ \text{odd}.
\end{cases}
\]
The square partial sums of \eqref{app:ms_series} numerically exhibit
$O(M^{-2})$ convergence; evaluating them in high-precision arithmetic and
extrapolating in $M$ gives
\[
\gamma_{\rm ref}=-0.006340363264 ,
\]
consistent with an independent computation on highly refined meshes
aligned with $\partial T_p$ and $\partial T_d$.

\subsection*{B.3 The convection--diffusion problem}

Here $\bff_2=\bg_2=\mathbf 0$ and $\G_N=\emptyset$, so
$\gamma=(g_1,u)_\O+(f_1,z)_\O-a(u,z)$ with $u$ and $z$ the exact fields of
Section~\ref{subsec:convection-diffusion}. All three terms are analytic on
$\O=(0,1)^2$ and are evaluated by tensor-product Gauss--Legendre
quadrature, giving $\gamma=0.1052871614333$.

\subsection*{B.4 The four L-shaped quantities of interest}

With $u_{\rm ex}=r^{2/3}\sin(2\theta/3)=\operatorname{Im}(x+\mathrm iy)^{2/3}$
and $\bsigma=-\nabla u_{\rm ex}$, each output of
Section~\ref{subsec:lshape} reduces to elementary integrals of
$\zeta^{2/3}$ or $\zeta^{-1/3}$ along the relevant edges or over the
relevant region, evaluated with the antiderivatives
$\tfrac35\zeta^{5/3}$, $\tfrac32\zeta^{2/3}$ and
$-\tfrac{9\mathrm i}{40}\zeta^{8/3}$. This gives the exact output values
\[
J_1(u_{\rm ex})=\frac{3\cdot2^{1/3}}{10}\bigl(3+\sqrt3\bigr),
\quad
J_2(u_{\rm ex})=-\frac{3\cdot2^{1/3}}{10}\bigl(1+\sqrt3\bigr),
\quad
J_4(u_{\rm ex})=\frac{3}{40}
\Bigl[-2+2^{2/3}+\sqrt3\bigl(4\cdot2^{1/3}-2-2^{2/3}\bigr)\Bigr],
\]
and, evaluating $-\tfrac{9\mathrm i}{40}\zeta^{8/3}$ at the four corners of
$\omega_0$,
\[
J_3(u_{\rm ex})=-\frac{9}{40}\operatorname{Re}
\Bigl[\bigl(-\tfrac12-\tfrac{\mathrm i}2\bigr)^{8/3}
     -\bigl(-1-\tfrac{\mathrm i}2\bigr)^{8/3}
     -\bigl(-\tfrac12-\mathrm i\bigr)^{8/3}
     +\bigl(-1-\mathrm i\bigr)^{8/3}\Bigr].
\]
Their decimal values are those of Table~\ref{tab:lshape-qoi}. 

\bibliographystyle{plain}
\bibliography{../bib/szhang}

\begin{thebibliography}{10}

\bibitem{BR:03}
Wolfgang Bangerth and Rolf Rannacher.
\newblock {\em Adaptive finite element methods for differential equations}.
\newblock Birkhauser, Basel, 2003.

\bibitem{BET:11}
Roland Becker, Elodie Estecahandy, and David Trujillo.
\newblock Weighted marking for goal-oriented adaptive finite element methods.
\newblock {\em SIAM Journal on Numerical Analysis}, 49(6):2451--2469, 2011.

\bibitem{BR:01}
Roland Becker and Rolf Rannacher.
\newblock An optimal control approach to a posteriori error estimation in
  finite element methods.
\newblock {\em Acta Numerica}, 10:1--102, 2001.

\bibitem{BG:09}
Pavel~B. Bochev and Max~D Gunzburger.
\newblock {\em {Least-Squares Finite Element Methods}}.
\newblock Applied Mathematical Sciences, 166. Springer, 2009.

\bibitem{BLP:97}
James~H. Bramble, Raytcho~D. Lazarov, and Joseph~E. Pasciak.
\newblock A least-squares approach based on a discrete minus one inner product
  for first order systems.
\newblock {\em Mathematics of Computation}, 66(219):935--955, 1997.

\bibitem{Bringmann:23}
Philipp Bringmann.
\newblock Review and computational comparison of adaptive least-squares finite
  element schemes.
\newblock {\em Computers \& Mathematics with Applications}, 172:1--15, 2024.

\bibitem{Cai:04}
Zhiqiang Cai.
\newblock Least-squares method.
\newblock Unpublished lecture notes of Purdue University, 2004.

\bibitem{CCLL:20}
Zhiqiang Cai, J.~Chen, M.~Liu, and X.~Liu.
\newblock Deep least-squares methods: an unsupervised learning-based numerical
  method for solving elliptic pdes.
\newblock {\em Journal of Computational Physics}, 420(109707), 2020.

\bibitem{CFZ:15}
Zhiqiang Cai, Rob Falgout, and Shun Zhang.
\newblock Div first-order system {LL* (FOSLL*)} least-squares for second-order
  elliptic partial differential equations.
\newblock {\em SIAM J. Numer. Anal.}, 53(1):405--420, 2015.

\bibitem{CLMM:94}
Zhiqiang Cai, Raytcho~D. Lazarov, Thomas~A. Manteuffel, and Stephen~F.
  McCormick.
\newblock {First order system least-squares for second order partial
  differential equations: Part I}.
\newblock {\em SIAM J. Numer. Anal.}, 31:1785--1799, 1994.

\bibitem{CMM:97}
Zhiqiang Cai, Tom Manteuffel, and Stephen~F. McCormick.
\newblock First-order system least squares for second-order partial
  differential equations: Part ii.
\newblock {\em SIAM J. Numer. Anal.}, 34(2):425--454, 1997.

\bibitem{CS:04}
Zhiqiang Cai and Gerhard Starke.
\newblock Least-squares methods for linear elasticity.
\newblock {\em SIAM J. Numer. Anal.}, 42:826--842, 2004.

\bibitem{CZ:10b}
Zhiqiang Cai and Shun Zhang.
\newblock Flux recovery and a posteriori error estimators: conforming elements
  for scalar elliptic equations.
\newblock {\em SIAM J. Numer. Anal.}, 48(2):578--602, 2010.

\bibitem{CDG:14}
Carsten Carstensen, Leszek Demkowicz, and Jay Gopalakrishnan.
\newblock A posteriori error control for {DPG} methods.
\newblock {\em SIAM J. Numer. Anal.}, 52(3):1335--1353, 2014.

\bibitem{LSQoI:14}
Jehanzeb~Hameed Chaudhry, Eric~C. Cyr, Kuo Liu, Thomas~A. Manteuffel, Luke~N.
  Olson, and Lei Tang.
\newblock Enhancing least-squares finite element methods through a
  quantity-of-interest.
\newblock {\em SIAM J. Numer. Anal.}, 52:3085--3105, 2014.

\bibitem{DPG:acta2025}
Leszek Demkowicz and Jay Gopalakrishnan.
\newblock The discontinuous {P}etrov--{G}alerkin method.
\newblock {\em Acta Numerica}, 34:293--384, 2025.

\bibitem{DGK:00}
Leszek Demkowicz, Jay Gopalakrishnan, and Brendan Keith.
\newblock The {DPG}-star method.
\newblock {\em Computers \& Mathematics with Applications}, 79:3092--3116,
  2020.

\bibitem{FPV:16}
Michael Feischl, Dirk Praetorius, and Kristoffer G~Van der Zee.
\newblock An abstract analysis of optimal goal-oriented adaptivity.
\newblock {\em SIAM J. Numer. Anal.}, 54(3):1423--1448, 2016.

\bibitem{GS:02}
Michael~B. Giles and Endre S{\"u}li.
\newblock Adjoint methods for pdes: a posteriori error analysis and
  postprocessing by duality.
\newblock {\em Acta Numerica}, 11:145--236, 2002.

\bibitem{Jiang:98}
Bo-nan Jiang.
\newblock {\em The Least-Squares Finite Element Method Theory and Applications
  in Computational Fluid Dynamics and Electromagnetics}.
\newblock Scientific Computation. Springer, 1998.

\bibitem{KAD:19}
Brendan Keith, Ali~Vaziri Astaneh, and Leszek~F Demkowicz.
\newblock Goal-oriented adaptive mesh refinement for {Discontinuous
  Petrov--Galerkin} methods.
\newblock {\em SIAM J. Numer. Anal.}, 57(4):1649--1676, 2019.

\bibitem{Ku:07}
JaEun Ku.
\newblock A remark on the coercivity for a first-order least-squares method.
\newblock {\em Numer. Methods Partial Differential Equations},
  23(6):1577--1581, 2007.

\bibitem{LZYZ:24}
Ziyan Li and Shun Zhang.
\newblock Non-intrusive least-squares functional a posteriori error estimator:
  Linear and nonlinear problems with plain convergence.
\newblock {\em Computers and Mathematics with Applications}, 191:275--295,
  2025.

\bibitem{LZ:24}
Yuxiang Liang and Shun Zhang.
\newblock Least-squares versus partial least-squares finite element methods:
  Robust a priori and a posteriori error estimates of augmented mixed finite
  element methods.
\newblock {\em Computers and Mathematics with Applications}, 184:45--64, 2025.

\bibitem{LZ:19}
Qunjie Liu and Shun Zhang.
\newblock Adaptive flux-only least-squares finite element methods for linear
  transport equations.
\newblock {\em Journal of Scientific Computing}, 84(2):26, 2020.

\bibitem{LZ:18}
Qunjie Liu and Shun Zhang.
\newblock Adaptive least-squares finite element methods for linear transport
  equations based on an {H(div)} flux reformulation.
\newblock {\em Comput. Methods Appl. Mech. Engrg.}, 366:113041, 2020.

\bibitem{MS:09}
Mario~S. Mommer and Rob Stevenson.
\newblock A goal-oriented adaptive finite element method with convergence
  rates.
\newblock {\em SIAM J. Numer. Anal.}, 47:861--886, 2009.

\bibitem{PG:00}
Niles~A Pierce and Michael~B Giles.
\newblock Adjoint recovery of superconvergent functionals from pde
  approximations.
\newblock {\em SIAM Review}, 42(2):247--264, 2000.

\bibitem{QZ:20}
Weifeng Qiu and Shun Zhang.
\newblock Adaptive first-order system least-squares finite element methods for
  second order elliptic equations in non-divergence form.
\newblock {\em SIAM J. Numer. Anal.}, 58(6):3286--3308, 2020.

\bibitem{Zhang:23}
Shun Zhang.
\newblock A simple proof of coerciveness of first-order system least-squares
  methods for general second-order elliptic {PDE}s.
\newblock {\em Computers \& Mathematics with Applications}, 130:98--104, 2023.

\end{thebibliography}
\end{document}